\documentclass[12pt]{amsart}

\usepackage{xy, graphicx,  hyperref}

\usepackage{amsmath}
\usepackage{amsthm}
\usepackage{amssymb}
\usepackage{amsfonts, mathrsfs}

\usepackage{tikz,verbatim}
\usetikzlibrary{positioning}

\xyoption{all} 

\makeatletter
\def\subsection{\@startsection{subsection}{3}%
	\z@{.9\linespacing\@plus.7\linespacing}{.1\linespacing}%
	{\normalfont\bfseries}}
\makeatother

\title[Characteristic Classes]{Characteristic Classes of Representations of Lie Groups}
\author{Rohit Joshi and Steven Spallone}

\newtheorem{theorem}{Theorem}
\newtheorem{corollary}{Corollary}
\newtheorem{lemma}{Lemma}

\newtheorem{prop}{Proposition}
\newtheorem{defn}{Definition}

\theoremstyle{definition}
\newtheorem{remark}{Remark}
\newtheorem{example}{Example}

\newcommand{\nc}{\newcommand}

\nc{\thm}{\theorem}
\nc{\cor}{\corollary}
\nc{\mc}{\mathcal}
\nc{\mb}{\mathbb}
\nc{\mf}{\mathfrak}
\nc{\ul}{\underline}
\nc{\ol}{\overline}
\nc{\N}{\mb N}
\nc{\R}{\mb R}
\nc{\Z}{\mb Z}
\nc{\Q}{\mb Q}
\nc{\C}{\mb C}
\nc{\Gm}{\mb G_m}

\nc{\rmt}{\res\mid_{T[2]}}

\nc{\dmo}{\DeclareMathOperator}

\nc{\mat}[4]{
	\begin{pmatrix}
		#1 & #2 \\
		#3 & #4
	\end{pmatrix}
}

\dmo{\Ker}{Ker} \dmo{\val}{val} \dmo{\ord}{ord}

\dmo{\ch}{ch}
\dmo{\id}{id}
\dmo{\odd}{odd}
\dmo{\sgn}{sgn}
\dmo{\st}{st}
\dmo{\Poly}{Poly}
\dmo{\Int}{Int}
\dmo{\diam}{diam}
\dmo{\Supp}{Supp}
\dmo{\dyn}{dyn}
\dmo{\Ad}{Ad}

\nc{\POLY}{\mc A}
\nc{\POLYZ}{\mc A_{\mathbb{Z}}}

\dmo{\BGL}{BGL}
\dmo{\BO}{BO}
\dmo{\Span}{Span}
\dmo{\ad}{ad}
\dmo{\sd}{sd}
\dmo{\HH}{H}
\dmo{\orth}{orth}
\dmo{\Sq}{Sq}
\dmo{\wt}{wt}
\nc{\beq}{\begin{equation*}}
	\nc{\eeq}{\end{equation*}}
\nc{\half}{\frac{1}{2}}
\nc{\hra}{\hookrightarrow}

\dmo{\Mod}{mod}
\dmo{\Orth}{O}
\dmo{\Or}{O}
\dmo{\core}{core}
\dmo{\res}{res}
\dmo{\lin}{lin}
\dmo{\vol}{vol}
\dmo{\Sp}{Sp}
\dmo{\diag}{diag}
\dmo{\SU}{SU}
\dmo{\SO}{SO}
\dmo{\SL}{SL}
\dmo{\GL}{GL}
\dmo{\gp}{gp}
\dmo{\PGL}{PGL}
\dmo{\U}{U}
\dmo{\SP}{SP}
\dmo{\Spin}{Spin}
\dmo{\GSp}{GSp}
\nc{\la}{\lambda}
\nc{\eps}{\varepsilon}
\dmo{\ev}{ev}

\nc{\lip}{\langle}
\nc{\rip}{\rangle}
\nc{\gm}{\gamma}
\nc{\ztwo}{\Z/2\Z}
\nc{\ms}{\mathscr}

\dmo{\Perm}{Perm}
\dmo{\Res}{Res}
\dmo{\Ind}{Ind}
\dmo{\tr}{tr}
\dmo{\Sym}{Sym}
\dmo{\reg}{reg}
\dmo{\End}{End}
\dmo{\trace}{trace}
\dmo{\Hom}{Hom}
\dmo{\Disc}{Disc}
\nc{\prd}{\prod_{\alpha \in R^+} \alpha}
\nc{\normi}{\sum_{\alpha \in R_i}\alpha^2}
\dmo{\simp}{sc}

\nc{\rmd}{\dfrac{d_{\mu}}{d_{\delta}}}
\address{Rohit Joshi, Shivachhatrapati Housing Society, Senadatta Peth, Pune-411030, Maharashtra, India}
\email{rohitsj2004@gmail.com}
\address{Steven Spallone, Indian Institute of Science Education and Research, Pune-411008, Maharashtra, India}
\email{sspallone@gmail.com}


\begin{document}

\maketitle

\begin{center} \today
\end{center}

\begin{abstract}
An irreducible representation of a reductive Lie algebra, when restricted to a Cartan subalgebra, decomposes into weights with multiplicity.
The first part of this paper outlines a procedure to compute symmetric polynomials (e.g., power sums) of this multiset of weights, as functions of the highest weight.
 
Next, let $G$ be a connected reductive complex algebraic group with maximal torus $T$. We express the restrictions of the Chern classes of irreducible representations of $G$ to $T$, as polynomial functions in the highest weight. We do the same for Stiefel-Whitney classes of orthogonal representations. 
\end{abstract}

\tableofcontents

\section{Introduction}

Let $\mf g$ be a complex reductive Lie algebra, with Cartan subalgebra $\mf t$, and $k \geq 0$.   We define the $k$th power sum of a representation $\phi$ of $\mf g$ as
\beq
P_k(\phi)=\sum_{\mu} m_\phi(\mu)\mu^k,
\eeq
where $\mu$ runs over the weights of $\phi$, and $m_\phi(\mu)$ denotes the multiplicity. It takes values in the symmetric algebra of $\mf t^*$. 
As is widely known, power sums generate the algebra of symmetric functions. 

The essential case is when $\phi$ is irreducible, say with highest weight $\la$, so $\phi=\phi_\la$. 
Determining each $m_{\phi_{\la}}(\mu)$  is subtle. On the other hand 
$P_0(\phi)=\deg \phi$ is explicitly given by the Weyl Dimension formula (WDF). 
Hence we apply WDF methods to compute the power sums, following ideas from 
 \cite[Ch. VIII, page 257, Exercise 8]{bourbaki}. Here we encounter the functions
\beq
F_k(\mu,\nu)=\sum_{w \in W} \sgn(w) \lip {}^w \mu,\nu \rip^k,
\eeq
for $\mu \in \mf t^*$ and $\nu \in \mf t$. The $F_k$ may be computed, in principle, through the theory of $W$-invariant polynomial functions on $\mf t$. 
(See Section \ref{fk.gen.sec}. Here $\delta$ is the usual half-sum of positive roots.)
This being done, the exponential generating function (EGF) of the sequence $P_k(\phi)$ is the quotient of the EGF of  $F_k(\la+\delta)$ by the EGF of $F_k(\delta)$. (See \eqref{EGFs}.) This gives a recursive formula for $P_k(\phi)$. We illustrate all this with several examples. 
 
Let $G$ be a  connected reductive complex algebraic group with Lie algebra $\mf g$, and $\pi$ a representation of $G$. (Throughout this paper all representations are algebraic.) Fix a maximal torus $T$ of $G$, with Lie algebra $\mf t$. Associated to $\pi$ are Chern classes (CCs) $c_k(\pi) \in \HH^{2k}(BG,\Z)$, where $BG$ is a classifying space for $G$. When $\pi$ is orthogonal, one also has Stiefel-Whitney classes (SWCs) $w_k(\pi) \in \HH^k(BG,\Z/2\Z)$; see Section \ref{app2} for details.

 Write $c_k^T(\pi)$ for the image of $c_k(\pi)$ under the restriction map \newline $\HH^{2k}(BG,\Z) \to  \HH^{2k}(BT,\Z)$, and similarly $w_k^T(\pi) \in \HH^k(BT,\Z/2\Z)$. 
Since $c_k^T(\pi)$ is essentially the $k$th elementary symmetric function of the weights of $\pi$, we can apply the power sum calculation from the first part of the paper.
 
Let $T[2]$ be the $2$-torsion subgroup of $T$.  We describe a surjective ring homomorphism $\varphi: \HH^*(BT,\Z) \to \HH^*(BT[2],\Z/2\Z)$
which takes $c_k^T(\pi)$ to $w_k^{T[2]}(\pi)$, when $\pi$ is orthogonal.

 Let $\pi_\la$ be the irreducible representation of $G$ with highest weight $\la$.  In Section \ref{summary.ch} we give a method to compute $c_k^T(\pi_\la)$ and $w_k^T(\pi_\la)$ in terms of $\la$. The method entails that their dependence on $\la$ is polynomial:

   \begin{thm} \label{intro.theorem} For all nonnegative integers $k$, the function $\lambda \mapsto c_k^T(\pi_\la) \in \mb \HH^{2k}(BT,\Z)$ is polynomial in $\la$ of degree no greater than $k(N+1)$, where $N$ is the number of positive roots.
  \end{thm}
  
  \bigskip
  
  For the case $k=2$, we give a  direct formula for $c_2^T(\pi_\la)$. To state it, write $\Ad$ for the adjoint representation of $G$ on $\mf g$, and write `$| \: |$' for the Killing form on $\mf t^*$.

  \begin{thm} \label{c2.thm} When $\mf g$ is simple, we have
  \beq
  c_2^T(\pi_\la)=\frac{|\la+\delta|^2-|\delta|^2}{\dim \mf g}  \cdot \deg \pi_\la  \cdot c_2^T(\Ad).
  \eeq
  
  \end{thm}

  
   We also observe that $c_2(\pi)$ vanishes iff $\pi$ is trivial (Proposition \ref{c^2.doesn't.vanish}).
We illustrate all this by computing $c_k^T(\pi_\la)$ and $w_k^T(\pi_\la)$ for $G=\SL_2(\C)$, $\SL_3(\C)$, and $\PGL_2(\C)$ and $k \leq 4$.  
We also compute $c_k^T(\pi_\la)$ for $\GL_2(\C)$ and $k \leq 2$.

When $G$ is one of $\GL_n(\C)$, $\SL_n(\C)$, $\SO_n(\C)$ or $\Sp_{2n}(\C)$, we give an alternate product expression for the \emph{total} SWC $w^T(\pi)=1+w_1^T(\pi)+ w_2^T(\pi)+ \cdots$ in terms of character values of $\pi$ at elements of order two, following the techniques of \cite{GJgln} and \cite{Malik.Spallone.SLn}. Also see Section \ref{CCchar} for an analogue for total CCs.

As an application, we give a new criterion for an orthogonal representation of $G$ to lift to the corresponding spin group. (Corollary \ref{spin.crit.26}). From it, one can recover our earlier result \cite[Theorem 1]{joshi}.

When $G$ is one of $\GL_n(\C)$, $\SL_n(\C)$, or $\Sp_{2n}(\C)$, the restriction map to the torus is injective, hence computing    $c_k^T$ (respectively, $w_k^T$)  is equivalent to determining $c_k$ (respectively, $w_k$).

 Now we describe the layout of this paper. In Section \ref{PS.prelim} we introduce the power sums $P_k(\phi)$; for $\mf g=\mf{sl}_2$ we compute them explicitly.
 In Section \ref{WCF.section} we apply the Weyl Character Formula to reduce the problem to computing the $F_k$. We develop this further in Section \ref{recursive.section}; in particular we show that $P_k(\phi_\la)$ is a polynomial function of $\la$, of degree at most $N+k$, and give a general expression for $P_2$. We also indicate how higher $F_k$ may be computed. To illustrate, we work out $P_2,P_3$, and $P_4$ for $\mf g=\mf{sl}_3$ in  Section \ref{example2}.
 
 Next, we turn to characteristic classes of Lie groups. In Section \ref{cc} we prove Theorems \ref{intro.theorem} and \ref{c2.thm}. We discuss the examples of $\GL_2$ and $\PGL_2$, and summarize the general CC calculations in Section \ref{summary.ch}. We compute the SWCs $w^T(\pi)$ of orthogonal representations in Section \ref{SWC.section}. 
 
We put a few  topological technicalities in the Appendix. In particular, we explain  a theory of SWCs for orthogonal representations of Lie
groups, compatible with the existing theory for real representations.

  \bigskip
  
\textbf{Acknowledgements.}  The first author was supported by a postdoctoral fellowship from NBHM (National Board of Higher Mathematics), India.   
 We thank Amit Hogadi and Neha Malik for helpful discussions.
 
\section{Preliminaries and Notation}  \label{prelim.not.section}

\subsection{Polynomials} \label{poly.prelim}
Let $V,W$ be finite-dimensional complex vector spaces, and  $f: V \to W$ a map. We say that $f$ is \emph{polynomial}, when for all linear maps $\nu: \C \to V$ and $\mu: W \to \C$, the composition $\mu \circ f \circ \nu$ is polynomial in the usual sense. Write `$\Poly(V,W)$' for the set of such polynomial functions. We may identify $\Poly(V,\C)$ with the symmetric algebra $\Sym (V^*)$, and  $\Poly(V,W)$ with $W \otimes \Sym(V^*)$.

If $b_1, \ldots, b_d$ is a basis of $V^*$, then  $\Poly(V,\C)$ is the polynomial algebra $\C[b_1, \ldots, b_d]$. 
  When $R$ is a commutative $\C$-algebra, we can identify $\Poly(V,R)$ with the polynomial algebra $R[b_1, \ldots, b_d]$. 

Suppose $L$ is a free abelian group of finite rank. We say that a function $f: L \to W$ is  \emph{polynomial}, when it is the restriction of
 a polynomial function $f: L \otimes_\Z \C \to W$. 
Finally, if $S \subset L$ is a subset, we say that $f: S \to W$ is \emph{polynomial}, when it is the restriction of a polynomial function on $L$. 

\subsection{Lie Algebras} \label{lie.alg.sec}

 Let $\mf g$ be a reductive complex Lie algebra with Cartan subalgebra $\mf t$.
  
Let  $\mb S = \Sym_\C(\mf t^*)$ and $\mb S^*=\Sym_\C(\mf t)$. 
For $k \geq 0$, put $\mb S^k=\Sym_\C^k(\mf t^*)$.

 Write $\Lambda$ for the weight lattice; following \cite[page 200]{Fulton}, this is the set of $\mu \in \mf t^*$ which are integer-valued on all coroots of $\mf t$ in $\mf g$. 
Define  $\mb S_{\Lambda}=\Sym_{\Z} \Lambda \subset \mb S$.
Fix a choice  $\Lambda^+ \subset \Lambda$ of dominant weights.
Let $W$ be the Weyl group of $\mf t$ in $\mf g$, and let $\sgn: W \to \{\pm 1\}$ be the usual sign function \cite[Section 24.1]{Fulton}.  Write $\delta$ for the half sum of the positive roots of $\mf t$ in $\mf g$. 

Let $K$ be the Killing form on $\mf t$; when $\mf g$ is semisimple, this defines an isomorphism $\sigma: \mf t \overset{\sim}{\to} \mf t^*$ by the formula $\lip {}^\sigma \nu ,\nu' \rip=K(\nu,\nu')$.
Write $K^\vee$ for the inverse form of $K$; it is defined by $K^\vee({}^\sigma \nu,{}^\sigma \nu')=K(\nu,\nu')$. 

 \begin{defn} For semisimple $\mf g$,
 define $q_2 \in \mb S$ by $q_2(\nu)=K(\nu,\nu)$, and when $q_2^\vee \in \mb S^*$ by $q_2^\vee(\mu)=K^\vee(\mu,\mu)$.
 \end{defn}

When $\mf g$ is simple, we have the following ``strange formula'' of Freudenthal and de Vries \cite{Freudenthal} or \cite[page 257]{bourbaki}:
\begin{equation} \label{strange}
24 q_2^\vee(\delta) =\dim \mf g.
\end{equation}

\subsection{The Algebra $\mc A$}

\begin{defn} Let  $\mc A=\Sym_\C(\mf t^* \oplus \mf t)$. 
 \end{defn}
Note that $\mc A= \mb S \otimes_{\C} \mb S^*$, hence it is bigraded and may be viewed in two ways: 
\begin{enumerate}
\item The algebra of complex-valued polynomials on pairs $(\mu,\nu) \in \mf t^* \times \mf t$.
\item The algebra of $\mb S$-valued polynomials on $\mf t^*$.
\end{enumerate} 

Given   $\mu \in \mf t^*$, there is an ``evaluation map''
\beq
\ev_\mu: \mc A \to \mb S
\eeq
given by $\ev_\mu(\bf G)(\nu)= \bf G(\mu,\nu)$.

 Here are a few flavors of ``degree'' for $f \in \mc A$, to refer to later:
  
\begin{defn} \label{degree.defns}
	Let $f = \sum\limits_{i,j}f_{i,j} \in \mc A$, with $f_{i,j} \in \Sym^i(\mf t^*) \otimes \Sym^j(\mf t)$.
	\begin{enumerate}
	\item $\deg^*(f)=\max \{ i \mid \exists j \text{ with } f_{ij} \neq 0\}$.
	\item  $	\deg(f)=\max \{ j \mid \exists i \text{ with } f_{ij} \neq 0\}$.
	\item $\ul{\deg}(f)=(\deg^*(f),\deg(f)) \in \Z^2$.
	\end{enumerate}
	\end{defn}

\section{Power Sums} \label{PS.prelim}
 
 Let $(\phi,V)$ be a (finite-dimensional, complex) representation of the Lie algebra $\mf g$. Given $\mu \in \mf t^*$, write $m_\phi(\mu)$ for the multiplicity of $\mu$ as a weight of $\phi$. 
 
 \begin{defn}
 For a nonnegative integer $k$, we define
 \begin{equation} \label{defn.pk.here}
 P_k(\phi)=\sum_{\mu} m_\phi(\mu) \mu^k \in \mb S^k,
 \end{equation}
 where $\mu$ ranges over $\mf t^*$. 
 \end{defn}
 
 (It is the $k$th power sum of the weights, counted with multiplicity.)
 By \cite[page 200]{Fulton}, each weight $\mu \in \Lambda$, so $P_k(\phi) \in  {\mb S}_{\Lambda}^k$.
 
 \begin{example} \label{first.ex} Let $\phi_\ell$ be the irreducible representation of $\mf g=\mf{sl}_2$ of degree $\ell+1$. 
 Let $\mf t$ be the diagonal Cartan subalgebra. Define $\mu_0 \in \mf t^*$ by $\mu_0(\diag(1,-1))=1$.
Note that $P_k(\phi_\ell)=0$ when $k$ is odd. Moreover $P_0(\phi_\ell)=\ell+1$.
 When $k$ is even, we have
 \beq
 P_k(\phi_\ell)=2(\ell^k+ (\ell-2)^k+ \cdots + (\epsilon+2)^k + \epsilon^k)\mu_0^k,
 \eeq
 where $\epsilon$ is the remainder of $\ell$ mod $2$. 
 
 Let $B_k(x)$ be the $k$th Bernoulli polynomial of degree $k$ \cite[Ch VI, Section 1]{Bou.FRV}.
 By the well-known Faulhaber's formula, for $k$ even we have
 \beq
 1^k+2^k + \cdots +\ell^k=\frac{B_{k+1}(\ell+1)}{k+1}.
 \eeq
For $\ell$ even, it follows that
 \begin{equation} \label{faul.poly}
  P_k(\phi_\ell)=\frac{2^{k+1}}{k+1} \cdot B_{k+1} \left(\frac{\ell+2}{2} \right)   \mu_0^k.
 \end{equation}
 In fact, \eqref{faul.poly} is also true for $\ell$ odd. One could deduce this from the identity  \cite[page 292]{Bou.FRV}
 \beq
 B_{k+1}(x)=2^k \left( B_{k+1} \left(\frac{x}{2} \right)  + B_{k+1}\left(\frac{x+1}{2} \right) \right).
 \eeq
However we will soon see (Section \ref{recursive.section}) that $\ell \mapsto  P_k(\phi_\ell)$ is polynomial. Then, since both sides of \eqref{faul.poly} are values of polynomials
for even $\ell$, equality holds for all $\ell$.  For $k=2$ this gives $P_2(\phi_\ell)=2 \dbinom{\ell+2}{3}$.

\end{example}
 
 Returning to general $\mf g$, we can easily understand the first two power sums.
  We have $  P_0(\phi)=\deg \phi$, and $ P_1(\phi)=\sum_\mu m_\phi(\mu) \mu \in \mf t^*$ is simply the trace, i.e., the composition
  \beq
  \mf t \hookrightarrow \mf g \overset{\phi}{\to} \mf{gl}(V) \overset{\tr}{\to} \C.
  \eeq
 In particular, $ P_1(\phi)=0$ whenever $\mf g$ is semisimple.

Let `$\ad$' be the adjoint representation. Then $P_2(\ad)=q_2$ by the formula 
\beq
K(\nu_1,\nu_2)=\sum_{\alpha} \lip \alpha,\nu_1 \rip \lip \alpha,\nu_2 \rip,
\eeq
from \cite[VIII, page 226]{Bou.Lie.4-6}.

 Suppose $\phi_1,\phi_2$ are representations of  $\mf g_1$ and $\mf g_2$. 
 Write $\phi=\phi_1 \boxtimes \phi_2$ for their exterior tensor product, a representation of $\mf g=\mf g_1 \oplus \mf g_2$. It is easy to see that
 \begin{equation} \label{ratz}
 P_k(\phi)=\sum_{i+j=k} \binom{k}{i,j} P_i(\phi_1) P_j(\phi_2).
  \end{equation}
 
 For a linear functional $\phi$  on an abelian Lie algebra $\mf t$, obviously $ P_k(\phi)=\phi^k$.
 Hence the essential case is when $\mf g$ is \emph{simple}. Our first goal in this paper is to determine $P_k(\phi)$ for irreducible representations as a function of the highest weight, when $\mf g$ is simple.
 
\begin{prop} \label{k.even.p.van} Suppose $k$ is even. Then $P_k(\phi)=0$ iff $\phi$ is trivial.
\end{prop}

\begin{proof} 
Note that $\mb S_{\Lambda}$ is isomorphic to a polynomial algebra over the integers. Hence if $x_1, \ldots, x_n \in \mb S_{\Lambda}$ with $\sum x_i^k=0$, then each $x_i=0$. So if $P_k(\phi)=0$, then the restriction of $\phi$ to $\mf t$ is trivial. But this restriction contains all the highest weights of the irreducible constituents of $\phi$. Hence $\phi$ is trivial.

\end{proof}

 We can similarly define the elementary symmetric functions of the weights of $\phi$. 
 
 \begin{defn} \label{e.defn.Here} Put $E(\phi)=\prod_{\mu} (1+ \mu)^{m_{\phi}(\mu)} \in \mb S_{\Lambda}$, and write $E_k(\phi) \in \mb S^k_{\Lambda}$ for the degree $k$ term of  $E(\phi)$.
 \end{defn}
 In particular $E_0(\phi)=1$. The recursive formula (2.11') from \cite[page 23]{macdonald} gives:
 \begin{equation} \label{better.cn.formula}
 nE_n(\phi)=\sum_{r=1}^n (-1)^{r-1} P_r(\phi) E_{n-r}(\phi).
 \end{equation}

A closed formula is given by Newton's Identity:
 \beq
  n! E_n(\phi)=\det \begin{pmatrix} 
 P_1(\phi) & 1 & 0 & \cdots & 0 \\
 P_2(\phi) & P_1(\phi) & 2 & \cdots & 0 \\
 \vdots & \vdots & \vdots & & \vdots \\
 P_{n-1}(\phi) & P_{n-2}(\phi) & & \cdots & n-1\\
 P_n(\phi) & P_{n-1}(\phi) & & \cdots & P_1(\phi) \\
 \end{pmatrix}.
 \eeq
 
 \begin{example} We continue with  Example \ref{first.ex}. From the above formulas, we deduce that $E_{k}(\phi_\ell) =0$ for $k$ odd,  $E_2(\phi_\ell)= -\binom{\ell+2}{3} \mu_0^2$,  and $E_4(\phi_\ell)= \frac{1}{3} \binom{\ell+2}{5} (5\ell+12) \mu_0^4$.  
 \end{example}

   \section{Weyl Character Formula} \label{WCF.section}
 
 In this section we follow ideas from the proof of the Weyl character formula, for example in \cite[Ch. VIII, \S 9]{bourbaki}. Indeed, our technique takes ideas from  [ibid, page 257, Exercise 8].
 
 Write $\mb S[[X]]$ for the formal power series ring over $\mb S$. 
 Let $\C[\mf t^*]$ be the complex group algebra over the additive group $\mf t^*$, say with basis $e_\mu$ for $\mu \in \mf t^*$.   We shall write the action of the Weyl group $W$ on $\mf t^*$ and associated objects with leading superscript, e.g., $\mu \mapsto {}^w \mu$.
 
 Define a $\C$-algebra homomorphism $f: \C[\mf t^*] \to \mb S[[X]]$ by
 \beq
 \begin{split}
 f(e_\mu) &=e^{\mu \cdot X} \\
 		&= 1+ \mu X + \half \mu^2 X^2 + \cdots \\
		\end{split}
		\eeq
As usual, put $J(\mu)=\sum_{w \in W}\sgn(w) \: {}^w e_\mu$, and for a representation $\phi$ of $\mf g$ put $\mf X_\phi=\sum_\mu m_\phi(\mu) e_\mu$. When $\phi$ is irreducible with highest weight $\la \in \Lambda^+$, we shall write $\phi=\phi_\la$.
Weyl's Character formula \cite[Ch. VIII, \S 9, Theorem 1]{bourbaki} may be written as:
\begin{equation} \label{WCForm}
J(\la+\delta)=J(\delta) \mf X_{\phi_\la}.
\end{equation}
Generally,
\beq
 \begin{split}
 f(\mf X_\phi) &= f \left( \sum_\mu m_\phi(\mu) \mu \right) \\
 			&= \sum_\mu m_\phi(\mu) e^{\mu X} \\
			&=\sum_\mu m_\phi(\mu) \sum_{i=0}^\infty \left( \frac{\mu^i}{i!}X^i \right) \\
			&=\sum_{i=0}^\infty \frac{X^i}{i!} \left( \sum_\mu m_\phi(\mu) \mu^i \right) \\
			&= \sum_{i=0}^\infty \frac{P_i(\phi)}{i!}X^i. \\
			\end{split}
			\eeq
Meanwhile,
\begin{align*}
	f(J(\mu))&= f \left(\sum_{w \in W}\sgn(w)\: {}^w \mu \right)\\
	&=\sum_{w \in W} \sgn(w) \: e^{{}^w \mu \:  X}\\
 	&= \sum_{w \in W} \sgn(w)\left(\sum_{i=0}^{\infty}\dfrac{ {}^w \mu^i}{i!} X^i \right)\\
 	&= \sum_{i=0}^{\infty} \dfrac{X^i}{i!} \sum_{w \in W} \sgn(w) \: ({}^w \mu)^i\\
 	&= \sum_{i=0}^{\infty} \dfrac{F_i(\mu)}{i!} X^i,
 \end{align*}
where we define
\begin{equation} \label{orig.Fk}
F_i(\mu)=\sum_{w \in W} \sgn(w) ({}^w \mu)^i \in \mb S^i.
\end{equation}
Hence applying $f$ to \eqref{WCForm} gives
\begin{equation} \label{EGFs}
 \sum_{i=0}^\infty \frac{F_i(\la+\delta)}{i!}X^i = \left( \sum_{j=0}^\infty \frac{F_j(\delta)}{j!}X^j \right) \left( \sum_{k=0}^\infty \frac{P_k(\phi_\lambda)}{k!}X^k \right).
 \end{equation}

In other words, the exponential generating function (EGF) for $F_i(\la+\delta)$ is the product of the EGF for $F_i(\delta)$ and the EGF for $P_k(\phi_\la)$. Therefore it is enough to evaluate the $F_i(\mu)$.

 \section{A Recursive Formula for Power Sums} \label{recursive.section}
 
 \subsection{Computing the $F_k$: First Steps}
 
 As mentioned in the introduction, the polynomials $F_k$ are crucial  for our calculation.
 First we``inflate'' them as follows.
 
 \begin{defn} Let ${\bf F}_k \in \mc A$ be given by
 \beq
 {\bf F}_k(\mu,\nu)=\sum_{w \in W} \sgn(w) \lip {}^w \mu,\nu \rip^k.
 \eeq
 \end{defn}
Given $\mu \in \mf t^*$, we have $\ev_\mu({\bf F}_k)=F_k(\mu)$. Since ${\bf F}_k \in \Sym^k(\mf t^*) \otimes \Sym^k(\mf t)$, either ${\bf F}_k=0$ or  $\ul \deg({\bf F}_k)=(k,k)$. 

Let $N$ be the number of roots of $\mf t$ in  $\mf g$. Define $d, d^\vee \in \mc A$ as 
\beq
d = \prod_{\alpha >0} \alpha \in \Sym^N(\mf t^*),
\eeq
 where $\alpha$ runs over positive roots and $$d^\vee= \prod_{\alpha^\vee > 0} \alpha^\vee \in \Sym^N(\mf t). $$
Moreover we have $d \cdot d^\vee \in \Sym^N(\mf t^*) \otimes \Sym^N(\mf t) \subset \mc A$; note that $\ul \deg(d)=(N,0)$ and $\ul \deg (d^\vee)=(0,N)$.
 
 \begin{prop} \label{ddividesfk} The product $d \cdot d^\vee$ divides ${\bf F}_k$ in $\mc A$.
 \end{prop}
 
 \begin{proof} The polynomial ${\bf F}_k$ is anti-W-invariant, in the sense that for all $w \in W$ we have
 \beq
 {\bf F}_k({}^w\mu,\nu)=\sgn(w) \cdot {\bf F}_k(\mu,\nu)=  {\bf F}_k(\mu,{}^w\nu).
 \eeq
 Therefore the divisibility follows from \cite[Proposition 3.13, Page 69]{RGCG}, taking $V=\mf t^* \oplus \mf t$ with reflection group   $W \times W$.
 \end{proof}
 
 \begin{defn} We put
 \beq
 {\bf F}_k'=\frac{{\bf F}_k}{d \cdot d^\vee} \in \mc A.
 \eeq
 \end{defn}
 
 By degree considerations, ${\bf F}_k=0$ when $k<N$. When $k \geq N$, we have ${\bf F}_k' \in \Sym^{k-N}(\mf t^*)^W \otimes \Sym^{k-N}(\mf t)^W$.

 When $\mf g$ is simple, the invariant subspace $\Sym^2(\mf t)^W$ is spanned by $q_2^\vee$, and similarly $q_2$ spans $\Sym^2(\mf t^*)^W$
(See for instance \cite[Table 1, page 59]{RGCG}.) 
 
 
\begin{prop} \label{F.calc}   For $\mf g$ simple, we have 
\begin{equation} \label{F_N.mu.eqn}
{\bf F}_N=N! \frac{d d^\vee}{d^\vee(\delta)},
\end{equation}
and
 \begin{equation} \label{F_N+2.mu.eqn}
{\bf F}_{N+2}=\binom{N+2}{2}  \frac{q_2 q_2^\vee}{\dim \mf g} {\bf F}_N.
\end{equation}
Moreover for $0 \leq k<N$ and $k=N+1$ we have ${\bf F}_k=0$.
\end{prop}
 
 \begin{proof}
 This is a restatement of \cite[Proposition 5]{joshi}.
	\end{proof}

   The same formula holds for ${\bf F}_N$ when $\mf g$ is reductive,. To go further, let $\mf z$ be the center of $\mf g$ and define
$u_{\mf z} \in \Sym(\mf z^* \oplus \mf z) \subset \mc A$ as the quadratic polynomial taking the pair $(\mu,\nu)$ to the evaluation $\lip \mu,\nu \rip$.
 By \cite[Proposition 6]{joshi}, we have
  \begin{equation}
 {\bf F}_{N+1}=(N+1) u_{\mf z} \cdot {\bf F}_N.
 \end{equation}
  When the derived algebra $\mf g'$ of $\mf g$ is simple, we have
   \begin{equation} \label{F_N+2.mu.eqn}
{\bf F}_{N+2}=\binom{N+2}{2}  \left( \frac{q_2 q_2^\vee}{\dim \mf g}+ u_{\mf z}^2 \right){\bf F}_N.
\end{equation} 
Here $q_2$ and $q_2^\vee$ are with respect to $\mf g'$.

 \begin{defn}
 We write $-1 \in W$, when the longest Weyl group element $w_0$ acts by $-1$ on $\mf t$. 
 \end{defn}
 
 \begin{prop} \label{-1.Fk} Suppose that $-1 \in W$. If $N+k$ is odd, then ${\bf F}_k=0$.
\end{prop}

\begin{proof} 
	The proposition follows from
	\beq
	\begin{split}
		(-1)^k {\bf F}_k(\mu,\nu) &={\bf F}_k(-\mu,\nu) \\
		&= {\bf F}_k({}^{w_0}\mu,\nu) \\
		&= \sgn(w_0) {\bf F}_k(\mu,\nu) \\
		&=(-1)^N {\bf F}_k(\mu,\nu). \\
	\end{split}
	\eeq \end{proof}
	Note that $F_N(\mu) \neq 0$ when $\mu$ is regular. Taking coefficients of $X^{N+i}$ in \eqref{EGFs} with $i \geq 1$ gives
 \begin{equation}\label{Fpk}
 \frac{F_{N+i}(\la+\delta)}{(N+i)!}=\sum_{k=0}^{i} \frac{P_k(\phi_\la)F_{N+i-k}(\delta)}{k! (N+i-k)!}.
 \end{equation}
 Solving for $P_i(\phi_\la)$ gives 
 
 \begin{equation}\label{recursive.pk}
 P_i(\phi_\la) =\frac{F_{N+i}(\la+\delta)-\sum_{k=0}^{i-1} \binom{N+i}{k} P_k(\phi_\la) F_{N+i-k}(\delta)}{\binom{N+i}{i} F_N(\delta)}.
 \end{equation}
Hence, we may recursively compute the power sums in terms of the $F_k$. Since the $F_k$ are polynomials, the function
\beq
\la \mapsto P_k(\phi_\la):=\sum_\mu m_{\phi_\la}(\mu)\mu^k,
\eeq
originally defined only for dominant $\la$, extends uniquely to a polynomial function from $\mf t^*$ to $\mb S^k$. We write this function as ${\bf P}_k \in \mc A$. To reiterate:

\begin{defn}
  Let ${\bf P}_k \in \mc A$ be the unique $\mb S^k$-valued polynomial whose value ${\bf P}_k(\la)$ at each dominant $\la \in \Lambda^+$ is given by
  ${\bf P}_k(\la)=P_k(\phi_\la)$. Here $P_k(\phi_\la)$ is defined by  \eqref{defn.pk.here}. 
Let  ${\bf E}_k \in \mc A$ be the unique $\mb S^k$-valued polynomial whose value ${\bf E}_k(\la)$ at each dominant $\la$ is  the degree $k$ term
of $E(\phi_\la)$. Here $E(\phi_\la)$ is given by Definition \ref{e.defn.Here}.
\end{defn}
Again, ${\bf P}_k$ and ${\bf E}_k$ are related through \eqref{better.cn.formula}. When ${\bf P}_k$ is viewed as a function on $\mf t^* \oplus \mf t$, we have
\beq
{\bf P}_k(\la,\nu)=\sum_\mu m_{\phi_\la}(\mu) \lip \mu,\nu \rip^k,
\eeq
for $\la$ a dominant weight.


 Let $\tau: \mc A \to \mc A$ be defined by $\tau(f)(\mu,\nu)=f(\mu+\delta,\nu)$, i.e., it is ``translation by $\delta$''. Recall the Weyl Dimension Formula
 \beq
 \deg \phi_\la=\frac{\prod  \lip \alpha^\vee, \la + \delta \rip}{\prod \lip \alpha^\vee, \delta \rip},
 \eeq
  where the products are taken over positive roots. Hence 
 \beq
 {\bf P}_0 =\frac{\tau(d^\vee)}{d^\vee(\delta)} \in \Sym^N \mf t.
 \eeq 
 We rewrite \eqref{recursive.pk} as 
 \begin{equation} \label{p_keqn}
 {\bf P}_i =\frac{\tau({\bf F}_{N+i})-\sum_{k=0}^{i-1} \binom{N+i}{k} F_{N+i-k}(\delta) {\bf P}_k }{\binom{N+i}{i} F_N(\delta)},
 \end{equation}
  for each $i \geq 1$. For $i=1$ this gives ${\bf P}_1=u_{\mf z} \cdot {\bf P_0}$, which is $0$ when $\mf g$ is semisimple.

 \begin{prop} \label{-1.p} If $-1 \in W$, then ${\bf P}_i=0$ for $i$ odd.
 \end{prop}
 
 \begin{proof} This follows from  \eqref{p_keqn} and Proposition \ref{-1.Fk}. 
 \end{proof}
 
 \begin{cor}\label{p2} For $\mf g$ simple, we have ${\bf P}_1=0$ and
 \beq
 {\bf P}_2=\left( \frac{\tau(q_2^\vee)-q_2^\vee(\delta)}{\dim \mf g} \right) q_2 {\bf P}_0.
 \eeq
 \end{cor}

 \begin{proof} All representations of $\mf g$ have   trace zero, so ${\bf P}_1 =0$. By  \eqref{p_keqn}, we have
 \beq
 {\bf P}_2=\frac{\tau({\bf F}_{N+2})-F_{N+2}(\delta) {\bf P}_0}{\binom{N+2}{2} F_N(\delta)}.
 \eeq
 Now 
 \beq
 \tau({\bf F}_{N+2})=\frac{(N+2)!}{48} \frac{d \tau(d^\vee)q_2 \tau(q_2^\vee)}{d^\vee(\delta)q_2^\vee(\delta)},
 \eeq
 so the result follows from \eqref{strange}. 
  
 \end{proof}
 
 Note that the polynomial $\tau(q_2^\vee)-q_2^\vee(\delta)$ evaluated at $\mu \in \mf t^*$ equals 
 \beq
 |\mu+\delta|^2-|\delta|^2=\lip \mu+2 \delta,\mu \rip.
 \eeq

  \begin{remark} 
 Please see \cite{joshi} for connections of this formula with the Casimir operator  and the Dynkin index.
  \end{remark}

We now estimate the degrees of these polynomials. The reader may review Definition \ref{degree.defns} for our terminology concerning   degrees.  
  
\begin{prop}\label{p.poly} We have $\deg {\bf P}_k \leq N+k$ and $\deg {\bf E}_k \leq k(N+1)$.

For $\mf g$ semisimple, we have 
\beq
\deg {\bf E}_k \leq \left[ \frac{k}{2} \right] N +k.
\eeq
\end{prop}

\begin{proof}
We have $\deg {\bf P}_0=\deg \tau(d^\vee)=\deg d^\vee=N$. For each $k$ either ${\bf F}_k=0$ or $\deg({\bf F}_k)=k$ by construction.   Inductively we deduce that each $\deg({\bf P}_k) \leq N+k$ from  the recursive formula \eqref{p_keqn}. For estimating $\deg {\bf E}_k$ we have ${\bf E}_0=1$, so $\deg {\bf E}_0=0$; moreover
\beq
n {\bf E}_n=\sum_{r=1}^n (-1)^{r-1} {\bf P}_r {\bf E}_{n-r}.
\eeq
by \eqref{better.cn.formula}.
Proceeding by induction we have, for $r \geq 1$,
\beq
\deg({\bf P}_r  {\bf E}_{n-r}) \leq (N+r)+(n-r)(N+1) \leq n(N+1),
\eeq
and the estimate for ${\bf E}_n$ follows.

When $\mf g$ is semisimple we have ${\bf E}_1={\bf P}_1=0$, and the same method will give the improved estimate.
\end{proof}

\subsection{Computing $F_k$ for $k >N+2$} \label{fk.gen.sec}
  
 From the above we see that computing symmetric polynomials in the roots reduces to calculating the ${\bf F_k}$.
 In this section we outline how this can be done; we will afterwards illustrate with $\mf {sl}_3(\C)$.
 
  As before, it is clear that ${\bf F}_k'$ is $W \times W$-invariant, hence lies in the subspace  \beq
 \left(\Sym^{k-N}(\mf t) \otimes \Sym^{k-N}(\mf t^*) \right)^W= \Sym^{k-N}(\mf t)^W \otimes  \Sym^{k-N}(\mf t^*)^W.
 \eeq
  
 We may view $W$ as a reflection group on $V=\mf t$. There is a well-known theory of $W$-invariant polynomials on $V$;   see the references on page 86 of \cite{RGCG} concerning the the problem of finding an explicit basis.
 
 To go a bit further, we show that ${\bf F}_k$ is invariant under one more operation. Let $\sigma: \mf t \to \mf t^*$ be the isomorphism coming from the Killing form. Since the Killing form is symmetric, we have
 $\lip {}^\sigma \nu_1, \nu_2 \rip = \lip {}^\sigma \nu_2, \nu_1 \rip$ for all $\nu_1,\nu_2 \in \mf t$. Moreover $\sigma$ is $W$-invariant.
 
  From $\sigma$ we obtain an involution $\mf t^* \oplus \mf t \to \mf t^* \oplus \mf t$ by
 \beq
 (\mu, \nu) \mapsto ({}^{\sigma}\nu, {}^{\sigma^{-1}} \mu)
 \eeq
 and this extends to an involution of $\Sym^*(\mf t^* \oplus \mf t)$, again denoted by `$\sigma$'.
 
 \begin{lemma} The polynomials ${\bf F}_k$ and ${\bf F}_k'$ are  $\sigma$-invariant.  
 \end{lemma}
 
 \begin{proof}
 For $\nu \in \mf t$ and $\mu \in \mf t^*$ we have
 \beq
 \begin{split}
 F_k({}^{\sigma}\nu,{}^{\sigma^{-1}}\mu) &= \sum_w \sgn(w) \lip {}^{w\sigma}\nu, {}^{\sigma^{-1}} \mu \rip^k \\
 &= \sum_w \sgn(w) \lip {}^\sigma \nu, {}^{\sigma^{-1} w^{-1}} \mu \rip^k \\
 &= \sum_w \sgn(w) \lip {}^{w^{-1}}\mu,\nu \rip^k    \\
 &= F_k(\mu,\nu).\\
 \end{split}
 \eeq
 (By definition $\lip {}^w \mu,\nu \rip= \lip \mu,{}^{w^{-1}} \nu \rip$ for all $w \in W$.)
 
 The function $\nu \mapsto d^\vee({}^\sigma \nu)$ is a degree $N$ anti-$W$-invariant polynomial, and therefore it is a constant multiple of $d(\nu)$, say
 \beq
 d^\vee({}^\sigma \nu)=c d(\nu)
 \eeq
 for some $c \neq 0$. Hence also $d^\vee(\mu)=c d({}^{\sigma^{-1}} \mu)$; this gives
 \beq
 d^\vee({}^\sigma \nu)d({}^{\sigma^{-1}}\mu)=d^\vee(\mu) d(\nu).
 \eeq
 Finally, 
 \beq
 \begin{split}
 F_k'({}^{\sigma}\nu,{}^{\sigma^{-1}}\mu) &=\frac{ F_k({}^{\sigma}\nu,{}^{\sigma^{-1}}\mu)}{ d^\vee({}^\sigma \mu)d({}^{\sigma^{-1}}\mu)} \\
				&= F_k'(\mu,\nu). \\
				\end{split}
				\eeq

 \end{proof}

 Here now is a method for explicitly computing the ${\bf F}_k$. First, find a basis $b_1, \ldots, b_\ell$ of $W$-invariant polynomials on $\mf t$, of degree $k-N$.  
 For $1 \leq i \leq j \leq \ell$, let $\beta_{ij}$ be the polynomial defined by
 \beq
 \beta_{ij}(\mu,\nu)=b_i(\nu)b_j({}^\sigma \mu)+ b_j(\nu)b_i({}^\sigma \mu)
 \eeq
 The $\beta_{ij}$ form a basis of polynomials on $\mf t^* \oplus \mf t$, which are doubly homogeneous, $W$-invariant of degree $k-N$, and $\sigma$-invariant. 
 Therefore ${\bf F}_k'$ is some linear combination of the $\beta_{ij}$, the coefficients of which can be determined by a few calculations.

 \section{Example: $\mf {sl}_3(\C)$}\label{example2}
 
 In this section we determine ${\bf P}_2, {\bf P}_3, {\bf P}_4,{\bf E}_2, {\bf E}_3$, and ${\bf E}_4$ for  $\mf g = \mf{sl}_3(\C)$.

 \subsection{Roots and Weights}
 
 Write $\mf t$ for the diagonal Cartan subalgebra of $\mf g$. Identify $\mf t$ with the vectors $v=(v_1,v_2,v_3) \in \C^3$ with $v_1+v_2+v_3=0$. For $1 \leq i \leq 3$, let $\hat e_i \in \mf t^*$ be the functional given by $\hat e_i(v)=v_i$.
 (Thus $\hat e_1+\hat e_2+\hat e_3=0$.)
 A positive system of roots for $\mf t$ in $\mf g$ is given by
 \beq
 R^+=\{\alpha_1,\alpha_2,\alpha_1+\alpha_2\},
 \eeq
 with $\alpha_1=\hat e_1-\hat e_2$ and $\alpha_2=\hat e_2-\hat e_3$. Thus $\delta=\hat e_1-\hat e_3=2\hat e_1+\hat e_2$.

 Fundamental   weights for $\mf t$ are  given by $\varpi_1=\hat e_1$ and $\varpi_2=\hat e_1+\hat e_2$.
 For integers $m,n \geq 0$, write $\phi_{m,n}$ for the irreducible representation of $\mf g$ whose highest weight is $m \varpi_1+n \varpi_2$.
 Then
 \beq
 \deg \phi_\la=\half (m+1)(n+1)(m+n+2).
 \eeq


The Killing form  $K$ on $\mf t \subset \C^3$ is $6$ times the restriction of the standard bilinear form (dot product), and the inverse form $K^\vee$ on $\mf t^*$
is determined by 
 \beq
 K^\vee(\hat e_i,\hat e_i)=\frac{1}{9}, \: \: K^\vee(\hat e_i,\hat e_j) =-\frac{1}{18} \text{ for $i \neq j$}.
 \eeq


 Corollary \ref{p2} gives $P_2(\phi_{m,n})$: We have 
 \beq
 \begin{split}
 P_2(\phi_{m,n}) &= \frac{\deg \phi_\la \lip \la+2 \delta,\la \rip}{8} q_2 \\
 &=\frac{1}{12}(m+1)(n+1)(m+n+2)(m^2+mn+n^2+3m+3n)(\hat e_1^2+\hat e_1\hat e_2+\hat e_2^2). \\
 \end{split}
 \eeq

\subsection{$F_6$ and $F_7$}
The crux of the computation of  $P_3(\phi_{m,n})$  is determining $F_6$. Now, the function $(\mu,\nu) \mapsto F_6(\mu,\nu)$ is antisymmetric in both $\mu$ and $\nu$. Therefore it is divisible by $d_\mu \cdot d_\nu$, and the quotient $F_6'(\mu,\nu)$ is $W \times W$-symmetric, i.e., ${\bf F}_6' \in \mc A^{W \times W}$. More precisely, for degree reasons, ${\bf F}_6' \in (\Sym^3 \mf t)^W \otimes (\Sym^3 \mf t^*)^W$. 

From character theory, say, we compute  that the subspace of $W$-invariant polynomials in $\Sym^3 \mf t$ is $1$-dimensional.
If we set $e_i'=e_i-\frac{1}{3}(e_1+e_2+e_3)$ for $1 \leq i \leq 3$, then $q_3^\vee=e_1'e_2'e_3'$ generates $(\Sym^3 \mf t)^W$.
Under the isomorphism $\sigma$ coming from the Killing form, we have $\sigma(e_i')=6\hat e_i$. So we  take $q_3=6^3\hat e_1\hat e_2\hat e_3 \in (\Sym^3 \mf t^*)^W$ as a basis.  
 
  Hence $q_3 \cdot q_3^\vee$ is a basis of $(\Sym^3 \mf t^*)^W \otimes (\Sym^3 \mf t)^W$. In particular, ${\bf F}_6'$ is a multiple of $q_3 \cdot q_3^\vee$, i.e.,
 \beq
 F_6(\mu,\nu)=c q_3(\nu) q_3^\vee(\mu) d^\vee(\mu) d(\nu),
 \eeq
 for some constant $c$ to be determined. We only need to compute both sides at values of $\mu$ and $\nu$ so that neither side is $0$.

 Let $\mu_0=3\hat e_1-\hat e_2-2\hat e_3$, and $\nu_0=\diag(3,-1,-2)$. Then $d^\vee(\mu_0)=d(\nu_0)=20$ and $q_3(\nu_0)=6^4$  and $q_3^\vee(\mu_0)=6$.
 $F_6(\mu_0,\nu_0)=5^2 \times 6^6$. Hence
 \beq
 c=\frac{5^2 \times 6^6}{20^2 \times 6^5}=\frac{3}{8},
 \eeq
 and so
\beq
 F_6(\mu,\nu)=\frac{3}{8} q_3(\mu) q_3^\vee(\nu) d^\vee(\mu) d(\nu).
 \eeq

Similarly, $\left(\Sym^4 \mf t\right)^W$ is $1$-dimensional, and proceeding as above gives
\begin{align*} 
F_7(\mu,\nu)=\frac{63}{16} q_2(\nu)^2q_2^\vee(\mu)^2 d(\nu) d^\vee(\mu). 
\end{align*}

Next,
\beq
  P_3(\phi_\la)=\dfrac{ F_6(\la+\delta)-\deg \phi_\la \cdot F_6(\delta) -15 P_2(\phi_\la) \cdot F_4(\delta)}{20 F_3(\delta)}.
\eeq

 \subsection{$F_8$ and $F_9$}
 
 By \cite[Table 1, page 59]{RGCG}, we have $\mb S^W=\C[q_2,q_3]$.
 For instance, the space of degree $5$ invariants is also $1$-dimensional, spanned by $q_2q_3$.
 Hence there is a constant $c \in \C$ so that
 \beq
 F_8(\mu,\nu)=c q_2^\vee(\mu)q_3^\vee(\mu)q_2(\nu) q_3(\nu)d(\nu) d^\vee(\mu).
 \eeq
 Plugging in $\mu_0$ and $\nu_0$ again computes this constant, and we obtain
  \beq
 {\bf F}_8 =\frac{3}{4} q_2^\vee q_3^\vee q_2  q_3 d d^\vee.
 \eeq

 The space of degree $6$ invariants is $2$-dimensional, spanned by $q_2^3$ and $q_3^2$.
 Hence there are constants $a_1,a_2,a_3$ so that 
 \beq
 {\bf F}_9=\left(a_1q_2^3 q_2^{\vee 3}+a_2 (q_2^3 q_3^{\vee 2}+ q_3^2 q_2^{\vee 3})+a_3q_3^2q_3^{\vee 2}\right)d  d^\vee.
 \eeq
 These constants can be determined by plugging in three pairs $\mu,\nu$ in general position, and doing so gives
 \beq
{\bf F}_9=\dfrac{3}{64}\left( 85 q_2^3  q_2^{\vee 3} -(q_2^3 q_3^{\vee 2}+q_3^2 q_2^{\vee 3} )+q_3^2q_3^{\vee 2}  \right) d d^\vee.
\eeq

 
 \subsection{Summary}

 To summarize:
 
 \begin{itemize}
 \item ${\bf F}_3=3d d^\vee$.
 \item ${\bf F}_4=0$.
 \item ${\bf F}_5=\frac{5}{4} q_2 q_2^\vee \cdot {\bf F}_3$.
 \item ${\bf F}_6= \frac{1}{8} q_3q_3^\vee  \cdot {\bf F}_3$.
 \item ${\bf F}_7=\frac{21}{16} q_2^2(q_2^\vee)^2  \cdot {\bf F}_3$.
\item ${\bf F}_8=\frac{1}{4} q_2 q_2^\vee q_3 q_3^\vee \cdot {\bf F}_3$.
 \item ${\bf F}_9=\frac{1}{64}\left[85 q_2^3 (q_2^\vee)^3 -  (q_2^3 q_3^{\vee 2}+q_3^2 q_2^{\vee 3} ) + q_3^2 (q_3^\vee)^2 \right]{\bf F}_3$.
 \end{itemize}

For the highest weights $\la= m \varpi_1 + n \varpi_2 = (m+n, n , 0)$, this gives:

   \begin{itemize}
\item    $ P_0(\phi_{m,n})= \deg \phi_\la = \half (m+1)(n+1)(m+n+2)$,
\item $ P_1(\phi_{m,n})=0$
\item $		 P_2(\phi_{m,n})=\frac{1}{6}\deg \phi_\la(m^2+mn+n^2+3m+3n)(\hat e_1^2+\hat e_1\hat e_2+\hat e_2^2),$
\item $ P_3(\phi_\la)= -\frac{1}{20} \deg \phi_\la(2m+n+3)(m+2n+3)(m-n)
		\hat e_1\hat e_2(\hat e_1+\hat e_2)$,
		\end{itemize}
		and
  \beq
	 P_4(\phi_{m,n})=\frac{1}{30} \deg(\phi_\la)(m^2+mn+n^2+3m+3n)(2(m^2+mn+n^2+3m+3n)-3)(\hat e_1^2+\hat e_2^2+\hat e_1\hat e_2)^2. \eeq
 
 The first few elementary symmetric functions $E_i$ are then
 \begin{itemize}
 	\item $E_1(\phi_{m,n}) = 0$,
 	\item $E_2(\phi_{m,n})= -\frac{1}{12}  P_0(\phi_{m,n})(m^2+mn+n^2+3m+3n) (\hat e_1^2+\hat e_1\hat e_2+\hat e_2^2)$,
 	\item $E_3(\phi_{m,n}) = \frac{1}{60} P_0(\phi_{m,n})(2m+n+3)(m+2n+3)(n-m) \hat e_1\hat e_2(\hat e_1+\hat e_2)$,
 	\item $E_4(\phi_{m,n}) =\frac{1}{8}( P_2(\phi_{m,n})^2 -2  P_4(\phi_{m,n}))$.
 \end{itemize}

 \section{Lie Groups} \label{cc}
 Now let $G$ be a complex connected reductive Lie group, with maximal torus $T$. 
 Write $\mf g$ and $\mf t$ for the Lie algebras of $G$ and $T$. Every Lie group $G$ has a classifying space, written $BG$. Briefly, this is a paracompact space which admits a contractible $G$-bundle. 
  Any two classifying spaces for $G$ are homotopy equivalent. When $G$ is discrete, classifying spaces are the same as the Eilenberg-Maclane spaces written $K(G,1)$ as in \cite{hatcher}.

 \subsection{The Algebra $\mb X$}
  Being  a complex torus, $T$   is  isomorphic to the $r$-fold product $(\C^\times)^r$ for some $r \geq 1$. The infinite complex projective space $\mb{CP}^{\infty}$
  is a classifying space for $\C^\times$, and so its $r$-fold product can be taken as $BT$. Write $X^*(T)$ as usual for the algebraic  homomorphisms from $T$ to $\C^\times$. (These are the irreducible algebraic representations of $T$.)
   
By Proposition \ref{first.cc.iso} in the Appendix, the first Chern class gives
 a bijection  \begin{equation} \label{c1.tori}
  X^*(T) \overset{\sim}{\to}  \HH^2(BT,\Z).
 \end{equation}
  
  \begin{defn}
  Put $\mathbb X=\Sym_{\Z} X^*(T)$.  
  \end{defn}
  
 Then \eqref{c1.tori} extends to a ring isomorphism 
 \begin{equation}\label{psiprime}  
\psi: \mb X \overset{\sim}{\to} \HH^{*}(BT,\Z).
 \end{equation} 
Differentiation $X^*(T) \to \mf t^*$   induces injections $X^*(T) \hookrightarrow \Lambda$, 
 and $\mb X \hookrightarrow \mb S_{\Lambda}$. These are isomorphisms when $G$ is semisimple and simply connected. (The image of $X^*(T)$ is the lattice `$\Lambda$' of \cite[Theorem 23.16]{Fulton}.)
Also put
 \beq
 Q_2=\psi(q_2)=\sum_{\alpha \in R} c_1(\alpha)^2.
 \eeq

 \subsection{Chern Classes} \label{jan.29.cc}
 
 The inclusion $T \hookrightarrow G$  gives rise to a restriction map $\HH^*(BG,\Z) \to \HH^*(BT,\Z)$ on cohomology.  Let $\pi$ be a representation of $G$, and write $d \pi$ for its differential, i.e., the corresponding Lie algebra representation. Write $c(\pi) \in \HH^*(BG,\Z)$ for the total Chern class (CC) of $\pi$. (See \cite[Section 2.6]{Benson}.)
 The $k$th graded piece is denoted $c_k(\pi)$. Chern class satisfy the following well-known properties:
 \begin{enumerate}
 \item If $\pi$ is a representation of $G$ and $\varphi: G' \to G$ is a homomorphism, then $\varphi^*(c(\pi))=c(\pi \circ \varphi)$, where $\varphi^*$ is the induced map on cohomology.
 \item If $\pi_1$ and $\pi_2$ are representations, then $c(\pi_1 \oplus \pi_2)=c(\pi_1) \cup c(\pi_2)$.
 \item If $k>\deg \pi$, then $c_k(\pi)=0$.
 \end{enumerate}
 
 \begin{defn} Let  $c^T(\pi) \in \HH^*(BT,\Z)$ be the restriction of $c(\pi)$ to $T$. 
 \end{defn}
 This is the total CC of the restriction of $\pi$ to $T$.  When the restriction map $\HH^*(BG,\Z) \to \HH^*(BT,\Z)$ is injective, computing $c^T(\pi)$   determines $c(\pi)$ as well. According to \cite{toda}, this happens for $G=\GL_n(\C)$, $\SL_n(\C)$, and $\Sp_{2n}(\C)$. This gives a convenient way to describe CCs, since the cohomology ring of $BT$ is polynomial.

Given a representation $\pi$ of $G$, write $d\pi$ for the corresponding Lie algebra representation of $\mf g$. Since every weight of $d \pi$ is the derivative of a weight of $\pi$, we know $P_k(d \pi)$ and $E_k(d\pi)$ are in $\mb X^k \subseteq {\mb S}^k_\Lambda$ for all $k$.

\begin{prop} \label{psi}  We have $c^T(\pi)=\psi(E(d\pi))$.
\end{prop}

\begin{proof}   For  $\chi \in X^*(T)$, let $m_\pi(\chi)$ be the multiplicity of $\chi$ in $\pi$. Viewing $E(d \pi)$ in $\mb X$, we write
 \beq
E(d\pi) = \prod_{\chi}(1+\chi)^{m_\pi(\chi)},
 \eeq
and so
\beq
\psi(E(d \pi)) = \prod_{\chi}(1+c_1^T(\chi))^{m_\pi(\chi)}.
\eeq
On the other hand, since $\pi|_T=\sum_\chi \chi^{\oplus m_\pi(\chi)}$, we have $c^T(\pi)=  \prod_{\chi} c^T(\chi)^{m_\pi(\chi)}$.
The result follows since $c^T(\chi)=1+c_1^T(\chi)$.
\end{proof}

 \begin{remark} This formulation is  close to \cite[Theorem 10.3(a)]{Borel.Hirz.58}.
 \end{remark}
 
Since $P_2(\ad)=q_2$ from Section \ref{PS.prelim}, we have
 \begin{equation}
 c_2^T(\Ad)=-\half Q_2.
 \end{equation}

   \begin{thm} \label{poly.thm} For all nonnegative integers $k$, the function $\lambda \mapsto c_k^T(\pi_\la) \in \mb X^k$ is polynomial in $\la$.
   We have $\deg {\bf c}_k \leq k(N+1)$.
When $G$ is semisimple,  $\deg {\bf c}_k \leq \left[ \frac{k}{2} \right] N +k$.
  \end{thm}
  
  \begin{proof} The function $\lambda \mapsto F_k(\la)$ is  polynomial by definition of $F_k$.
 Inductively we see from \eqref{recursive.pk}  that $\lambda \mapsto   P_k(\phi_\la)$ is polynomial.
  Next, from \eqref{better.cn.formula} we find that $E_k(\phi_\la)$ is polynomial in $\la$.
  But now $c_k^T(\pi_\la)=\psi(E_k( \phi_\la))$ by  Proposition \ref{psi}, hence it is polynomial.
  The degree assertions follow from Proposition \ref{p.poly}.
  
  
  \end{proof}
  
Let ${\bf c}_k \in \mc A$ be the resulting polynomial; it is the degree $k$ polynomial with values in $\mb X \cong \HH^{2k}(BT,\Z)$ with the property that
${\bf c}_k(\la)=c^T_k(\pi_{\la})$ for all $\la \in \Lambda^+$. By the above   ${\bf c}_k=\psi \circ {\bf E}_k$.
  
 
 
  \begin{prop} \label{c^2.doesn't.vanish} Suppose $G$ is semisimple. A representation $\pi$ of $G$ is trivial iff $c_2^T(\pi)=0$.
\end{prop}
 
 \begin{proof}
 Since $G$ is semisimple, $E_1(d \pi)=0$ and so $E_2(d \pi)=-P_2(d \pi)$.
  Therefore if $c_2^T(\pi)=0$, then $P_2(d \pi)=0$, so $d \pi$ is trivial by Proposition \ref{k.even.p.van}. Hence $\pi$ is also trivial. 
  \end{proof}

 To illustrate, we transform the formulas for various $E_k(\phi)$ in the previous section into formulas for CCs.
 
 \begin{example} \label{sl2.chern.ex} Let $G=\SL_2(\C)$. The character $\chi_1 \in X^*(T)$ given by $\chi_1 \left( \diag(a,a^{-1}) \right)=a$ has derivative $\mu_0 \in \mf t^*$.
 Also write $e=c_1(\chi_1) \in \HH^*(BT,\Z)$ for the corresponding cohomology class.
 
 Let $\pi_\ell$ be the irreducible representation of $G=\SL_2(\C)$ of degree $\ell+1$, then 
 \begin{equation} 
{\bf P}_k(\ell)=\frac{2^{k+1}}{k+1} \cdot B_{k+1} \left(\frac{\ell+2}{2} \right) e^k   \in   \HH^{2k}(BT,\Z).
\end{equation}

For  $k=2$, this gives ${\bf c}_2(\ell)={\bf E}_2(\ell)=-\binom{\ell+2}{3}e^2$.
Similarly, 
\beq
{\bf c}_4(\ell)=\frac{1}{3} \binom{\ell+2}{5}(5 \ell+2) e^4.
\eeq
 
 \end{example}

  \begin{thm} \label{c_2.simp} When $\mf g$ is simple, we have
\beq
{\bf c}_2=-\left( \frac{\tau(q_2^\vee)-q_2^\vee(\delta)}{2 \dim \mf g} \right)Q_2 \cdot {\bf P}_0.
\eeq
 
  \end{thm}
  
  \begin{proof} This follows from  Corollary \ref{p2}, and the combinatorial identity $2 {\bf E}_2={\bf P}_1^2- {\bf P}_2$.
  \end{proof}

 Another way of writing this is
  \beq
  c_2^T(\pi_\la)=\frac{|\la+\delta|^2-|\delta|^2}{\dim \mf g}  \cdot \deg \pi_\la  \cdot Q_2.
  \eeq
Since $c_2^T(\Ad)=-\half Q_2$, this is equivalent to Theorem \ref{c2.thm} in the Introduction.

 \begin{remark} Let $\hat \HH^*(BG,\Q)$ be the completed rational cohomology ring of $BG$. In our notation, the restriction of the Chern character  $\ch(\pi)$ \cite[page 109]{Weibel.K} to $T$ is given by
 \beq
 \ch^T(\pi)=\sum_{k=0}^\infty \frac{P_k(\pi)}{k!} \in \hat \HH^*(BT,\Q).
 \eeq
 \end{remark}

 \subsection{Isogeny}
 
 Suppose $p: G' \to G$ is an isogeny, and $T' <G'$ is the preimage of a maximal torus $T$ in $G$. Then $X^*(T)$ injects into $X^*(T')$, hence the induced map $\HH^*(BT,\Z) \to \HH^*(BT',\Z)$ is injective. Here $c^T(\pi) \mapsto c^{T'}(\pi \circ p)$. Therefore, if we are only interested in $c^T(\pi)$, it is enough to compute $c^{T'}(\pi)$ for simply connected $G'$.  
 
 \begin{example}  \label{pgl2.ex} Let $G=\PGL_2(\C)$. For $\ell$ even, let $\ol \pi_\ell$ be the representation of $G$ descended from the $\pi_\ell$ of Example \ref{sl2.chern.ex}.
 Let $T$ be the diagonal torus of $G$ and $T'$ the diagonal torus of $\SL_2$. 
 Also  let $\ul e=c_1(\ul \chi_1)$, where $\ul \chi_1(\diag(a,b))=ab^{-1}$. 
 Note that under $X^*(T) \to X^*(T')$, $\ul e \mapsto 2e$. 
 
 From the above we have
 \beq
 P_k(\ol \pi_\ell)= \frac{1}{2(k+1)} B_{k+1} \left(\frac{\ell+2}{2} \right) \ul e^k \in \mb X^k.
 \eeq
For the case of $k=2$ this gives
 \beq
 {\bf c}_2(\ell)=-\frac{1}{4} \binom{\ell+2}{3} \ul e^2,
 \eeq
 since this maps to $-\binom{\ell+2}{3} e^2$. 
 For the case of $k=4$ this gives
 \beq
 {\bf c}_4(\ell)=\frac{1}{48} \binom{\ell+2}{5} (5 \ell+12) \ul e^4.
 \eeq
 \end{example}
 
\begin{example} Let $G=\GL_2(\C)$. The diagonal torus $T$ of $G$ has its group of characters $X^*(T)$ generated by $e_1,e_2$, where $e_k$ projects to the $k$th diagonal entry. Using the same notation for their first Chern classes  gives $\HH^*(BT,\Z)=\Z[e_1,e_2]$,  The dominant weights of $G$ are then $\la_{m,n}=me_1+ne_2$ with $m \geq n \geq 0$.
 
 Write $\tilde G=\SL_2(\C) \times \C^\times$, and let $\tilde T$ be the product of the diagonal torus in $\SL_2(\C)$ with $\C^\times$. Then $\HH^*(B\tilde T,\Z)=\Z[e,t]$, where $t=c_1(\chi_1)$.
 Define an isogeny $\rho: \tilde G \to G$ by $\rho(A,z)=Az$. For the induced map $\rho^*$ on cohomology, we have $\rho^*(e_1)=e+t$ and $\rho^*(e_2)=-e+t$.
 
 For $m,n$ as above, let $\Pi_{m,n}$ be the irreducible representation of $G$ with highest weight $\la_{m,n}$. Then $\pi_{m,n} =\Pi_{m,n} \circ \rho$ is the external tensor product $\chi_{m+n} \boxtimes \pi_{m-n}$.  
 
  This gives $c_1(\pi_{m,n})=(m-n+1)(m+n)t$ and
  \beq
 c_2(\pi_{m,n})=-\binom{m-n+2}{3} e^2 + \binom{m-n+1}{2} (m+n)^2 t^2,
 \eeq
 
 Then $c_k^T(\Pi_{m,n})$ for $k=1,2$ are the unique cohomology classes with $\rho^*(c_k^T(\Pi_{m,n}))=c_k^T(\pi_{m,n})$. 
 Write ${\bf c}_k(m,n)=c_k^T(\Pi_{m,n})$. Some light calculation gives 
 \beq
 {\bf c}_1(m,n)=\half(m-n+1)(m+n)(e_1+e_2)
 \eeq
  and
 \beq
\begin{split}
{\bf c}_2(m,n) &=\left[\dfrac{(m-n+1)(m-n)(m+n)^2}{8}- \frac{1}{4} \binom{m-n+2}{3}\right](e_1^2+e_2^2) \\
		& +\left[\dfrac{(m-n+1)(m-n)(m+n)^2}{4}+  \half \binom{m-n+2}{3}\right](e_1e_2). \\
		\end{split}
\eeq
\end{example}

 \subsection{Synopsis of Chern Class Calculation} \label{summary.ch}
 
We summarize the above for the reader's convenience.
 
Let $G$ be a connected complex reductive group, with maximal torus $T$, and fix an integer $k \geq 0$. Let $N$ be the number of positive roots of $T$ in $G$. There is a polynomial ${\bf c}_k \in \mc A$, whose value at $\la \in \Lambda^+$ is the restriction of $c_k(\pi_\la)$ to $T$.
 In particular it takes values in $\HH^{2k}(BT,\Z)$, which we can view as the polynomial algebra $\Sym_\Z X^*(T)$.
 
It is necessary to compute the polynomials ${\bf F}_i  \in \mc A$ for $N \leq i \leq N+k$. This calculation is discussed in Section \ref{recursive.section}.  When the Lie algebra $\mf g$ is simple, Proposition \ref{F.calc} gives a formula for ${\bf F}_N$ and ${\bf F}_{N+2}$, and ${\bf F}_{N+1}=0$. For $i \geq N+3$, see Section \ref{fk.gen.sec} for a procedure. Note that if $-1 \in W$, then ${\bf F}_i=0$ unless $i \equiv N \mod 2$ by Proposition \ref{-1.Fk}.
 
 Next, the power sums ${\bf P}_i \in \mc A$ are computed  by the recursive formula \eqref{recursive.pk}.
 After this, one computes ${\bf E}_k \in \mc A$ using   \eqref{better.cn.formula}. 
 
 Irreducible representations of $\mf g$ coming from the Lie group $G$ correspond to highest weights in $X^*(T) \cap \Lambda^+$. 
 When ${\bf E}_k$ is restricted to $X^*(T)$, it takes values in ${\mb X}^k$. Composing this with $\psi$ defined by \eqref{psiprime} gives the Chern class, in other words
  ${\bf c}_k=\psi \circ {\bf E}_k$.
 
 \bigskip

 When $G$ is reductive, its Lie algebra $\mf g$ is the direct sum of its center $\mf z$ and semisimple Lie algebras $\mf g_i$. An irreducible representation of $\mf g$ decomposes into an external tensor product of irreducible representations of $\mf z$ and $\mf g_i$. 
 One may compute the $P_k$ from the formula \eqref{ratz}, and then carry out the above procedure. 
 
 Alternatively, multiplication gives an isogeny $\rho: D(G) \times C \to G$, where $D(G)$ is the derived group of $G$ and $C$ is the connected component of the center of $G$. (See \cite[Section 14.2]{Borel.LAG}.) Composing $\pi$ with $\rho$ is then an external tensor product, and $\rho^*$ is injective. When $D(G)$ is a familiar group, this may be more direct. This was illustrated in the $\GL_2$ example above.

  \section{Stiefel-Whitney Classes} \label{SWC.section}
 
 \subsection{Orthogonality} \label{ortho.subs}
We say that $\pi$ is \emph{orthogonal} when there is a symmetric nondegenerate bilinear $G$-invariant form on $V$. An orthogonal linear character is necessarily quadratic, meaning it takes values in $\{ \pm 1\}$. 

For a representation $(\pi,V)$ of $G$, write $S(\pi)$ for the direct sum $V \oplus V^\vee$, where $V^\vee$ is the dual of $V$. It is regarded as an orthogonal representation via the bilinear form $B((v,\vartheta),(v',\vartheta'))=\lip v,\vartheta' \rip+\lip v',\vartheta \rip$.

We say an orthogonal representation $\pi$ is \emph{orthogonally irreducible}, or an \emph{OIR}, when it cannot be decomposed into a direct sum of orthogonal representations. All OIRs are either irreducible, or of the form $\pi=S(\sigma)$, where $\sigma$ is an irreducible representation which is not orthogonal. In particular, when $G$ is connected abelian, then all nontrivial OIRs are of the form $S(\chi)$, where $\chi \neq 1$ is a linear   character of $G$. Any orthogonal representation decomposes uniquely into a direct sum of OIRs.

\subsection{Definition}
 
 Now let $(\pi,V)$ be an orthogonal complex representation of $G$ of degree $n$.
 Then $\pi$ induces a map on cohomology
 \beq
 \pi^*: \HH^*(\BO_n(\C),\Z/2\Z) \to \HH^*(BG,\Z/2\Z).
 \eeq
 The ring $\HH^*(\BO_n(\C),\Z/2\Z)$ is polynomial on classes $w_1, \ldots, w_n$ with $|w_k|=k$. One defines $w_k(\pi)=\pi^*(w_k)$ (see Appendix \ref{app2}), and then
 \beq
 w(\pi)=1+w_1(\pi)+ \cdots + w_n(\pi).
 \eeq

 \begin{defn} Write $w^T(\pi) \in \HH^*(BT,\Z/2\Z)$ for the restriction of $w(\pi)$ to $T$.
 \end{defn}
 
We say that $T$ \emph{detects} the mod $2$ cohomology (of $G$) when the restriction map $\HH^*(BG,\Z/2\Z) \to \HH^*(BT,\Z/2\Z)$ is injective.
 In this case, computing $w^T(\pi)$ determines $w(\pi)$ as well. According to \cite{toda}, $T$ detects the mod $2$ cohomology, when $G=\GL_n(\C)$, $\SL_n(\C)$, $\SO_n(\C)$, or $\Sp_{2n}(\C)$.  In general $T$ detects the quadratic mod $2$ cohomology by Proposition \ref{quad.tor.detect}.

 \subsection{The Algebra $ \ol {\mb  X}$. }

  Let   $T$ be a complex torus, and $T[2]$ its $2$-torsion subgroup. 
Write  $T[2]^\vee$ for $\Hom(T[2], \{\pm 1\})$.

\begin{defn}
Put $\ol {\mb  X}=\Sym(T[2]^\vee)$.
\end{defn}

 Restriction gives a surjection $\Res: X^*(T) \twoheadrightarrow T[2]^\vee$; this extends   to a surjection $\Sym(\Res):\mb X \twoheadrightarrow 
 \ol{\mb X}$, whose kernel is $2 \mb X$. By Proposition \ref{first.swc.isomo}, the first SWC gives an isomorphism $w_1: T[2]^\vee \overset{\sim}{\to} \HH^1(BT[2],\Z/2\Z)$. This extends to an isomorphism
 \begin{equation} \label{psiprimebar} 
 	\ol{\psi}: \ol{\mb X} \overset{\sim}{\to} \HH^*(BT[2],\Z/2\Z).
 \end{equation} 
 
 The composition
\beq
\HH^2(BT,\Z) \overset{\psi^{-1}}{\to} X^*(T) \overset{\Res}{\to} T[2]^\vee \overset{\ol \psi}{\to} \HH^1(BT[2],\Z/2\Z)
\eeq
therefore extends to a surjective ring homomorphism
\beq
\varphi: \HH^*(BT,\Z) \to \HH^*(BT[2],\Z/2\Z).
\eeq
Note that $\varphi$ does not preserve degree.

From   \eqref{psiprime}, \eqref{psiprimebar}, and the above $\varphi$, we obtain the commutative diagram
\begin{equation} \label{comdiag} 
\centerline{\xymatrix{ 
				\mb X \ar[r]^{\Sym(\Res)} \ar[d]^{\psi} & \ol{\mb X} \ar[d]^{\ol{\psi}} \\
				\HH^*(BT,\Z) \ar[r]^{\varphi \: \: \: \: \quad} & \HH^*(BT[2],\Z/2\Z).\\
		}}
\end{equation}

\begin{lemma} \label{iota.injects}
The restriction map $\iota: \HH^*(BT,\Z/2\Z) \to \HH^*(BT[2],\Z/2\Z)$ is injective.
\end{lemma}

\begin{proof} The essential case is $T=\C^\times$. In this case $\HH^*(BT,\Z/2\Z)$ is a polynomial ring $\Z/2\Z[x]$, where $x \in \HH^2(BT,\Z/2\Z)$ corresponds to the nontrivial double cover of $\C^\times$ by itself. The restriction of this cover to $\mu_2 \subset \C^\times$ is nontrivial, so $\iota(x) \in \HH^2(BT[2],\Z/2\Z)$ is nonzero. Of course $\HH^*(BT[2],\Z/2\Z)$ is a polynomial ring $\Z/2\Z[v]$, with $v$ of degree $1$. Therefore $\iota(x)=v^2$ and the lemma follows. 
\end{proof}

 \begin{prop} \label{watch} If $\pi$ is an orthogonal representation of $T$, then  
	\begin{equation} \label{main}
		 w^{T[2]}(\pi)=\varphi(c(\pi)).
	\end{equation}
\end{prop}

\begin{proof}

Since $\varphi$ is multiplicative, it suffices to check \eqref{main} for $\pi$ orthogonally irreducible. 
So we may take $\pi=S(\chi)= \chi \oplus \chi^{-1}$ for a  linear character $\chi$.  We have 
\begin{align*}
	\iota^*(w(S(\chi))) &= \iota^*(w(\chi \oplus \chi^\vee)) \\
	&= w(\omega \oplus \omega), 
	\end{align*}
	where $\omega=\chi|_{T[2]}$, a quadratic linear character. Since $\omega$ is orthogonal, this equals
	\beq
	w(\omega) \cup w(\omega) =(1+w_1(\omega)) \cup (1+w_1(\omega)).
	\eeq
	Meanwhile, we have
\begin{align*}
	\varphi(c(S(\chi))) &=\varphi(c(\chi \oplus \chi^{-1}))\\
	&= \varphi((1 + c_1(\chi))(1+c_1(\chi^{-1})))\\
	&=(1+w_1(\omega)) \cup (1+w_1(\omega)),
\end{align*}
since by definition,
\begin{equation} \label{vaccine}
\varphi(c_1(\chi))=w_1(\omega).
\end{equation}
This completes the proof.
\end{proof}

\begin{remark}
Note that under the isomorphism \eqref{psiprimebar}, $w^{T[2]}(\pi)$ corresponds to the product
\beq
\prod_{\omega} (1+\omega)^{m_\pi(\omega)},
\eeq
where the product is taken over $T[2]^\vee$, and $m_\pi(\omega)$ is the multiplicity of $\omega$ in the restriction of $\pi$ to $T[2]$. This form of the SWC is very close to \cite[Theorem 11.3]{Borel.Hirz.58}.
\end{remark}

 Now let $G$ be connected complex reductive, with a maximal torus $T$. For an orthogonal representation $\pi$ of $G$, write $w^T(\pi)$ for the restriction of $w(\pi)$ to $T$.
  Equation \eqref{main} computes $w^T(\pi)$, since $c^T(\pi)$ is determined by Proposition  \ref{psi}.  
  
  \begin{example} We continue with Example \ref{sl2.chern.ex} for $G=\SL_2(\C)$. Note that $\pi_\ell$ is orthogonal iff $\ell$ is even. Putting $\varphi(e)=v$, we compute that
$w_2^T(\pi_\ell)=  \binom{\ell+2}{3}v^2$ and $w_4^T(\pi_\ell)=\frac{1}{3} \binom{\ell+2}{5}(5 \ell+2)  v^4$. It is easy to see that $\binom{\ell+2}{3}$ and $\frac{1}{3} \binom{\ell+2}{5}(5 \ell+2)$ are always even for $\ell$ even, hence $w_2^T(\pi_\ell)=w_4^T(\pi_\ell)=0$.
Hence $w^T(\pi_\ell)=1$ modulo terms of degree $6$ and higher. (In Proposition \ref{IOR.sl2}  we will see $w(\pi_\ell)=1$.)

For $\ell$ odd, we have $w_2(S(\pi_\ell))=0$, and $w_4(S(\pi_\ell))=\binom{\ell+2}{3}v^4$, which is $v^4$ for $\ell \equiv 1 \mod 4$, and vanishes for $\ell \equiv 3 \mod 4$. For $\ell \equiv 1 \mod 4$ we have
\beq
w^T(\pi_\ell)=1+ v^4+O(v^6).
\eeq

\end{example}

\begin{example} We continue with Example \ref{pgl2.ex}  for $G=\PGL_2(\C)$.
Note that $\varphi(\ul e)=\ul v$, where $\ul v$ is the nonzero member of $\HH^1(BT[2],\Z/2\Z)$.
We compute that $w_2^T(\pi_\ell)=\frac{1}{4} \binom{\ell+2}{3} \ul v^2$ and 
 \beq
 w_4^T(\pi_\ell)=\frac{1}{48} \binom{\ell+2}{5} (5 \ell+12) \ul v^4.
 \eeq
 We have
 \beq
 w_2^T(\pi_\ell) = \begin{cases} & 0 \text{ when $\ell \equiv 0,6 \mod 8$} \\
 & \ul v^2 \text{ when $\ell \equiv 2,4 \mod 8$}, \\
 \end{cases}
 \eeq
 and
  \beq
 w_4^T(\pi_\ell) = \begin{cases} & 0 \text{  when $\ell \equiv 0,2,4,14 \mod 16$} \\
 & \ul v^2 \text{ when $\ell \equiv 6,8,10,12 \mod 16$.} \\
 \end{cases}
 \eeq

In particular, for $\SL_2(\C)$, $w_2^T(\pi_4)=0$ but for $\PGL_2(\C)$ we have $w_2^T(\pi_4) \neq 0$.

\end{example}

\subsection{Spinoriality}

Let $G$ be a connected reductive complex algebraic group, and $(\pi,V)$ be a complex orthogonal representation of $G$.
 Since $G$ is connected we have $\pi: G \to \SO(V)$. Then $\pi$ lifts to the corresponding spin group $\Spin(V)$ (i.e., $\pi$ is \emph{spinorial})  iff $w_2(\pi)=0$. 

 \begin{prop} Let $k$ be a positive integer. Then $w_k^T(\pi)=0$ iff $c_k^T(\pi) \in 2\HH^{2k}(BT,\Z)$.
 \end{prop}
 
 \begin{proof}
  By Lemma \ref{iota.injects} and Proposition \ref{watch} we have:
\beq
\begin{split}
w_k^T(\pi)=0 & \Leftrightarrow w_k^{T[2]}(\pi)=0 \\
 & \Leftrightarrow  c_k^T(\pi) \in \ker \varphi. \\
 \end{split}
\eeq
 But this kernel is $2\HH^{*}(BT,\Z)$.
 \end{proof}

 \begin{cor} The representation $\pi$ is spinorial iff $c_2^T(\pi) \in 2\HH^{4}(BT,\Z)$.
 \end{cor}
\begin{proof} 
By Proposition \ref{quad.tor.detect} and the above, we have:
\beq
\begin{split}
\pi \text{ is spinorial} & \Leftrightarrow w_2(\pi)=0 \\
			 & \Leftrightarrow w_2^T(\pi)=0 \\
			  & \Leftrightarrow c_2^T(\pi) \in 2\HH^{4}(BT,\Z). \\
			  \end{split}
			  \eeq
\end{proof}

Combining this with Proposition \ref{c_2.simp} gives:

  \begin{cor} When $\mf g$ is simple, $\pi_\la$ is spinorial iff 
  \beq
  \dfrac{\deg \pi_\la  \lip \lambda + 2\delta,\lambda \rip}{4\dim \mf g} \cdot q_2 \in \Sym^2_{\Z} X^*(T).
  \eeq
 \end{cor}
 
 Actually, it is only the ``$2$ part'' of the coefficient of $q_2$ that matters. To be more precise, let $j(\pi)=-\ord_2 \left( \dfrac{\deg \pi_\la  \lip \lambda + 2\delta,\lambda \rip}{4\dim \mf g} \right)$. 
 
   \begin{cor} \label{spin.crit.26} When $\mf g$ is simple, $\pi_\la$ is spinorial iff 
  \beq
 2^{-j(\pi)}\cdot q_2 \in \Sym^2_{\Z} X^*(T).
  \eeq
 \end{cor}
 
 \begin{proof} Write  
 \beq
 \frac{\deg \pi_\la  \lip \lambda + 2\delta,\lambda \rip}{4\dim \mf g} =\frac{m}{2^jn},
 \eeq
 with $m,n$ odd. Note that $\dfrac{m}{2^{j-1}n}q_2=c_2^{T}(\pi_\la) \in \Sym^2(X^*(T))$. Let $A,B \in \Z$ so that $\frac{1}{2n}=\frac{A}{2}+\frac{B}{n}$.
 Then $\pi_\la$ is spinorial iff $\dfrac{m}{2^{j}n}q_2 \in  \Sym^2(X^*(T))$. But
 \beq
 \begin{split}
 \frac{m}{2^{j}n}q_2 		&= \frac{m}{2^{j-1}} \left( \frac{A}{2}+\frac{B}{n} \right)q_2 \\
		&= \frac{Am}{2^j}q_2 +B c_2^{T}(\pi_\la). \\
		\end{split}
		\eeq
 Hence the criterion is that $\dfrac{Am}{2^j}q_2 \in \Sym^2 X^*(T)$. Since $A,m$ are odd, we apply another Bezout identity argument to 
 deduce that $\pi_\la$ is spinorial iff $\dfrac{1}{2^j} q_2 \in \Sym^2 X^*(T)$.

 \end{proof}

 \begin{remark} By \cite[Proposition 7]{joshi}, the representation $\pi$ is spinorial iff for all $\nu \in X_*(T)$, the quantity
 \beq
  \dfrac{\deg \pi_\la  \lip \lambda + 2\delta,\lambda \rip}{4\dim \mf g} \cdot q_2(\nu)
  \eeq
 is an integer. But any \emph{homogeneous} quadratic polynomial on $X_*(T)$ taking integer values   lies in $\Sym^2 X^*(T)$ by \cite[Proposition 2, page 177]{bourbaki}.
  \end{remark}

 \section{Total Stiefel-Whitney Classes}
 
 In this section we give an alternate approach to determining $w^T(\pi)$ for an orthogonal representation $\pi$, when $G$ is one of the groups 
 $\GL_n(\C)$, $\SL_n(\C)$, $\SO_n(\C)$, or $\Sp_{2n}(\C)$. In each of these cases, the Weyl group $W$ has a  subgroup $\Sigma < W$ isomorphic to the symmetric group $\Sigma_r$, where $r$ is the rank of $T$, and   there is an isomorphism
 \begin{equation} \label{strong ind}
T \cong (\C^\times)^r
\end{equation}
 of  Lie groups which preserves the $\Sigma_r$-action. 
 
 We will determine $w^T(\pi)$ as a certain factorization, whose exponents are combinations of character values of $\pi$ at elements of order $2$.

\subsection{A Factorization of $w^T(\pi)$}
Recall from Lemma \ref{iota.injects} that $T[2]$ is a detecting subgroup for the mod $2$ cohomology of the maximal torus $T$. So we may describe $w^{T}(\pi)$ through its image in the polynomial algebra $\HH^*(BT[2],\Z/2\Z)=(\Z/2\Z)[v_1, \ldots, v_r]$. 

The restriction $\pi|_{T[2]}$ is an $S_r$-invariant representation of the elementary abelian $2$-group $T[2]$. For $0 \leq i \leq r$, write $b_i \in T[2]$ for the element corresponding to
$(\underbrace{-1,\ldots, -1}_{i \text{ times }},1,1,\ldots)$ under \eqref{strong ind}. 

\begin{example} For $G=\SL_3$, we may take $\eqref{strong ind}$ to be  $\diag(a,b, (ab)^{-1}) \leftrightarrow (a,b)$. Here $\Sigma \cong \Sigma_2$ permutes the first two diagonal entries. Hence $b_0=\diag(1,1,1)$, $b_1=\diag(-1,1,-1)$ and $b_2=\diag(-1,-1,1)$.
\end{example}

A typical member of $\HH^1=\HH^1(BT[2],\Z/2\Z)$ may be expressed as $v=a_1v_1+ \cdots +a_rv_r$, with $a_i \in \Z/2\Z$. Let $|v|$ be the number of $i$ with $a_i \neq 0$.

According to \cite[Proposition 2]{GJgln}, we have
\beq
w^{T[2]}(\pi)=\prod_{k=1}^r \left( \prod_{v \in \HH^1 : |v|=k} (1+v) \right)^{m_k(\pi)},
\eeq
where 
\beq
m_k(\pi)=\frac{1}{2^r} \sum_{i=0}^r A_{i,k} \cdot \chi_\pi(b_i).
\eeq
Here $A_{i,k}$ is the coefficient of $x^i$ in the expression $(1-x)^{k}(1+x)^{r-k}$.  (See also \cite[Section 4]{Malik.Spallone.SLn}.)


\subsection{Example: $\SL_2$}

For $G=\SL_2(\C)$, $b_1=\diag(-1,-1)=-1$, and we have
\beq
w(\pi)=(1+v)^{m_1(\pi)},
\eeq
where $m_1(\pi)=\half (\deg \pi-\chi_\pi(-1))$.

 \begin{prop} \label{IOR.sl2} For $G=\SL_2(\C)$, and $\pi$ an IOR of $G$, we have $w(\pi)=1$.
 \end{prop}
 \begin{proof}
When $\pi=\pi_\ell$ is orthogonal, we must have $\ell$ even and  $-1$ acts trivially, so $m_1(\pi)= 0$.
\end{proof}

On the other hand, for $\ell$ odd, we have $m_1(S(\pi_\ell))=2m_1(\pi_\ell)=2(\ell+1)$, and so
\beq
w(S(\pi_\ell))=(1+v^2)^{\ell+1}.
\eeq
In particular, $w(S(\pi_1))=(1+v^2)^2=1+v^4$, so $w_4(S(\pi_1))=v^4$. For any $k$, we have $w_{2k}(S(\pi_\ell))=\binom{\ell+1}{k} v^{2k}$, and the odd SWCs vanish.

 \subsection{Example: $\SL_3$}

  Take $G = \SL_3(\C)$ and $T$ the diagonal torus. 
The irreducible representation $\pi_{m,n}$ is orthogonal when $m=n$. Put 
\beq
\mc D=(1+v_1)(1+v_2)(1+v_1+v_2)=1+ v_1^2+v_2^2+ v_1v_2 + v_1^2 v_2+ v_1 v_2^2,
\eeq
as in \cite[Section 3.4]{Malik.Spallone.SLn}.
Then $w^{T[2]}(\pi)= \mc D^{m(\pi)}$, where
\beq
m(\pi)=\frac{1}{4} (\deg \pi-\chi_\pi(\diag(-1,-1,1))).
\eeq
Recall from \cite[Section 24.2]{Fulton} that
\beq
\chi_\la(\diag(t_1,t_2,t_3))=S_{\la}(t_1,t_2,t_3),
\eeq
where $S_\la$ is the Schur polynomial. The same reference gives the Giambelli formula for $S_\la$ in terms of complete symmetric polynomials $H_p(t_1,t_2,t_3)$. So it is enough to compute $H_p(1,1,1)$ and $H_p(-1,-1,1)$. Of course, $H_p(1,1,1)=\binom{p+2}{2}$; a combinatorial argument gives that
\beq
H_p(-1,-1,1)= \left\lfloor \frac{p}{2} \right\rfloor +1.
\eeq
For $\la=(2n,n,0)$, this gives $S_\la(1,1,1)=\deg \pi_\la=(n+1)^3$ and
\beq
S_\la(-1,-1,1)= \begin{cases} n+1, & \text{ $n$ even} \\
0 & \text{ $n$ odd} \\
\end{cases}.
\eeq
Therefore
\beq
m(\pi_{n,n})=
\begin{cases}
\frac{1}{4} (n+1)^3, & \text{ when $n$ is odd} \\
\frac{1}{4} n(n+1)(n+2), & \text{ when $n$ is even} \\
\end{cases}.
\eeq

For example, let $\pi$ be the adjoint representation. Then $n=1$ and  $m(\pi)=2$. Hence $w^{T[2]}(\pi)=\mc D^2$, and
\beq
w_4^{T[2]}(\pi)=v_1^4+v_2^4+v_1^2 v_2^2 \neq 0.
\eeq

\bigskip

Next we compute $w(S(\pi))$ for $\pi=\pi_{m,n}$. Note that $\chi_{S(\pi)}(g)=2 \chi_{\pi}(g)$ whenever $\chi_{\pi}(g)$ is real.
  A similar calculation to the above gives  
\beq
 	\chi_{\pi_{m,n}}(-1,-1,1)= \begin{cases}
		0  \quad &\text{ if } m \text{ odd and } n \text{ odd }\\\\
		-\dfrac{m+1}{2} \quad  &\text{ if } m \text{ odd and } n \text{ even }\\\\
		-\dfrac{n+1}{2}  \quad   &\text{ if } m \text{ even and } n \text{ odd }\\\\
		\dfrac{m+n}{2} +1 \quad  &\text{ if } m \text{ even and } n \text{ even }\\\\
	\end{cases} 
\eeq
so that
\beq
m(S(\pi_{m,n})) = \begin{cases}
	\frac{1}{4} (m+1)(n+1)(m+n+2) & \text{if $m$ odd and $n$ odd } \\
	 \frac{1}{4} (m+1)((n+1)(m+n+2) + 1) & \text{ if $m$ odd and $n$ even}\\
	\frac{1}{4}(n+1)((m+1)(m+n+2)+1) & \text{if $m$ even and $n$ odd}\\
	\frac{1}{4}(m+n+2)((m+1)(n+1)-1) & \text{if } m \text{ and } n \text{ are even.}\\
\end{cases}
\eeq

 \begin{remark} Note in all cases   $m(\pi)$ is even; this is necessary because $G$ is simply connected and so
 \beq
 w_2^{T[2]}(\pi)=m(\pi) (v_1^2+v_2^2+ v_1v_2)
 \eeq
 must vanish.
 
 \end{remark}


 \subsection{Remark on Total Chern Classes}\label{CCchar}
In this section, we outline an analogue of the above procedure for computing total Chern Classes, in the case of $G=\GL_n(\C)$ for simplicity of notation.
The restriction of a representation $\pi$ of $G$ to its diagonal torus $T$ is invariant under the Weyl group, and we may try to exploit this to understand the total CC of $\pi$. 

So let $T=(\C^\times)^n$, and $\ul m=(m_1, \ldots, m_n) \in \Z^n$. Define $\omega_{\ul m} \in X^*(T)$ by
 \beq
 \omega_{\ul m}(a_1, \ldots, a_n)=a_1^{m_1} \cdots a_n^{m_n}.
 \eeq
 
 We put
 \beq
 \sigma_{\ul m}=\sum_{\ul m'}\omega_{\ul m'},
 \eeq 
 where $\ul m'$ runs over the $S_n$-orbit of $\ul m$.
 Then the indecomposable $S_n$-invariant algebraic representations of $T$ are precisely the $\sigma_{\ul m}$ with $m_1 \leq \cdots \leq m_n$. 
  We have
 \beq
 c(\sigma_{\ul m})= \prod_{\ul m'} \left( 1+  \ul m' \cdot t \right),
 \eeq
 where the product runs over the $S_n$-orbit of $\ul m$. 
 
 Returning to $\pi$, knowing the multiplicity of each  $\sigma_{\ul m}$ in the restriction of $\pi$ to $T$ would give a product formula for the total CC.
Indeed,
 \begin{equation} \label{rainy}
 c(\pi)=\prod_{\ul m} c(\sigma_{\ul m})^{[\sigma_{\ul m},\pi]},
 \end{equation}
 where the product runs over increasing $\ul m$.

 Two versions of character orthogonality provide formulas, in principle, for these multiplicities. By restricting to the maximal compact abelian subgroup $(S^1)^n < T$, orthogonality gives:
 \beq
 [\sigma_{\ul m},\pi]=\int_{(S^1)^n} \chi_\pi(t) \overline{\omega_{\ul m}(t)} dt.
 \eeq
 For another version, let $A_N$ be the cyclic subgroup of $T$ generated by $a_{N}=(\zeta_N, \zeta_{N^2}, \ldots, \zeta_{N^n})$.
 Then
 \beq
[\sigma_{\ul m},\pi]=\frac{1}{N^n} \sum_{a \in A_N} \chi_\pi(a) \overline{\omega_{\ul m}(a)},
\eeq 
 for sufficiently large $N$. When $\pi$ is irreducible with highest weight $\la$, we may take $N \geq 2 |\la|$. For \eqref{rainy}, only needs to compute this multiplicity when $|\ul m| \leq 2|\la|$.

 \section{Appendix} \label{appendix}

\subsection{Some Algebraic Topology}

\begin{prop} \label{max.compact.equiv}  Let $K$ be a maximal compact subgroup of a Lie group $G$. The inclusion $K \hookrightarrow G$ induces a homotopy equivalence between $BK$ and $BG$.
\end{prop}

Though this is evidently well-known \cite[Page 79]{toda}, we sketch a proof for the reader's convenience.

\begin{proof}
	By Whitehead's Theorem \cite[Theorem 4.5, page 346]{hatcher}, it is enough to show that the inclusion induces isomorphisms on all homotopy groups 
	$\pi_n(BK) \to \pi_n(BG)$, for $n \geq 0$.

	To the fibre bundles  $K \to EK \overset{\rho_K} \to BK$ and  $G \to EG \overset{\rho_G} \to BG$ are associated long exact sequences, 
	\beq
	\xymatrix{ 
		\cdots \rightarrow \pi_n(K)  \ar[r] \ar[d] & \pi_n(EK) \ar[r] \ar[d] & \pi_n(BK) \ar[r]^{\delta_n \: \: \: \: \: \: \: \: }  \ar[d] & \pi_{n-1}(K) \ar[d] \rightarrow \cdots \\
		\cdots \rightarrow \pi_n(G) \ar[r]& \pi_n(EG) \ar[r] & \pi_n(BG) \ar[r]^{\delta_n  \: \: \: \: \: \: \: \:} & \pi_{n-1}(G) \rightarrow \cdots \\
	}
	\eeq
	with downward maps induced from the inclusions, and the diagram commutes by \cite[Theorems 4.3 and 4.41]{hatcher}.
	Since $EK$ and $EG$ are contractible, the boundary maps $\delta_n$ are isomorphisms. Consider the square:
	
	\begin{center}
		\centerline{\xymatrix{ 
				\pi_n(BK) \ar[r]^{\delta_n} \ar[d] & \pi_{n-1}(K) \ar[d] \\
				\pi_n(BG) \ar[r]^{\delta_n} & \pi_{n-1}(G)\\
		}}
	\end{center}

Since the inclusion $K \to G$ is a homotopy equivalence \cite[Theorem 2.2 page 257]{HGSN}, the right downward map is an isomorphism. Therefore the left vertical map is an isomorphism, as desired.
	
\end{proof}

\begin{prop} \label{first.cc.iso} Let $T$ be a complex torus. The first Chern class gives an isomorphism 
\beq
X^*(T) \overset{\sim}{\to} \HH^2(BT,\Z).
\eeq
\end{prop}

\begin{proof} 
It is enough to prove this for $T=\C^\times$. We have the diagram
\begin{center}
		\centerline{\xymatrix{ 
				X^*(\C^\times) \ar[r]^{c_1} \ar[d] & \HH^2(B \C^\times,\Z) \ar[d] \\
				\Hom(S^1,\C^\times) \ar[r]^{c_1} & \HH^2(BS^1,\Z) \\
		}}
	\end{center}
where the downward maps are by restriction.  The restriction map from $X^*(\C^\times)$ is clearly an isomorphism, and the other downward map is an isomorphism
by Proposition \ref{max.compact.equiv}. Let $\chi_{\bullet}: S^1 \to \C^\times$ be the usual inclusion; this generates the cyclic group $\Hom(S^1,\C^\times)$.
For $ES^1 \to BS^1$ we may take the fibration $S^\infty \to \C P^\infty$. 
By \cite[Proposition 7.1, page 96]{Husemoller}, the associated bundle $S^\infty[\chi_\bullet]$ over $ \C P^\infty$ is the tautological bundle $\gm_1$.
By definition of Chern classes \cite[page 249]{Husemoller}, the class $c_1(\gm_1)$ generates $\HH^2(BT,\Z)$.
Hence the bottom horizontal map is an isomorphism, and therefore so is the upper horizontal map.
\end{proof}

\begin{prop} \label{first.swc.isomo} Let $E$ be an elementary abelian $2$-group. The first SWC gives an isomorphism
\beq
\Hom(E,\{ \pm 1\})  \overset{\sim}{\to} \HH^2(BE,\Z/2\Z).
\eeq
\end{prop}

\begin{proof} This proof is similar to that of the previous proposition. We may assume $E$ is cyclic of order $2$. Let   $\sgn$ be the nontrivial quadratic character of $E$.  By \cite[Proposition 7.1, page 96]{Husemoller}, the associated bundle $S^\infty[\sgn]$ over $\R P^\infty$ is the tautological bundle $\gm_1$. By definition of SWCs \cite[page 248]{Husemoller},  $w_1(\gm_1)$ generates $\HH^1(BE,\Z)$.
 \end{proof}

\begin{prop}  \label{quad.tor.detect} Let $G$ be a connected reductive Lie group, $T$ a maximal torus of $G$, and $A$ an abelian group. Then the restriction map
	\beq
	\HH^2(BG,A) \to \HH^2(BT,A)
	\eeq
	is injective.
\end{prop}

In other words, $T$ detects the degree $2$ cohomology of $G$.

\begin{proof}
By Proposition \ref{max.compact.equiv}, we may assume $G$ is compact.

	Since $G$ is connected, we have $\pi_1(BG) = \pi_0(G) = \{0\}$ which implies $\HH_1(BG)=\{0\}$ and an isomorphism $\pi_2(BG) \cong  \pi_1(G)$. By Hurewicz Theorem (\cite[Theorem 4.32]{hatcher}), we have a natural isomorphism
	$\HH_2(BG) \cong \pi_1(G)$. 
	The natural map
	\beq
	\HH^2(BG,A) \to \Hom_{\Z}(\pi_1(G),A)
	\eeq
	is an isomorphism by the Universal Coefficient Theorem.
	
	Therefore it is enough to see that
	\beq
	\Hom_{\Z}(\pi_1(G),A) \to  \Hom_{\Z}(\pi_1(T),A)
	\eeq
	is an injection. But this follows from the fact \cite[Theorem 7.1]{BrokerDieck} that the induced homomorphism $\pi_1(T) \to \pi_1(G)$ is surjective.
\end{proof}

\subsection{Two Kinds of Stiefel-Whitney Classes}\label{app2}

Let $G$ be a Lie group. When $V$ is a real vector space, and $\pi: G \to \GL(V)$ is a homomorphism,   we say $(\pi,V)$ is a \emph{real representation} of $G$. In this case, SWCs $w_k(\pi)$  have been defined in \cite[Section 2.6]{Benson}, by means of the more well-known theory of SWCs of real vector bundles.  

When $G$ is compact, real representations of $G$ are equivalent to complex orthogonal representations of $G$ through complexification \cite[Chapter II, Section 6]{BrokerDieck}.
However this may fail for $G$ noncompact: for example the standard representation of $\GL_n(\R)$ on $\R^n$ does not admit an invariant nondegenerate symmetric bilinear form and the standard representation of $\Or_n(\C)$ on $\C^n$ does not admit a real form. 

Nonetheless, in this section, we define SWCs $w_k^{\Or}(\pi) \in \HH^k(BG,\Z/2\Z)$ for \emph{complex orthogonal} representations of $G$ in a natural way. Proofs   have been omitted for brevity.

Let $\Or_n$ be the usual compact orthogonal group (matrices with real entries whose transpose is their inverse).
 Inside $\HH^*(\BO_n,\Z/2\Z)$ is $w=1+w_1+\cdots+w_n$, with each $|w_k|=k$. (See for instance \cite[Section 2.6]{Benson} for a precise definition of $w_k$.)

 Write $\iota^{\R}: \Or_n  \hookrightarrow \GL_n(\R)$ and $\iota^{\Or}: \Or_n \hra \Or_n(\C)$ for the inclusions. Since they are homotopy equivalences, the induced maps
 $B \iota_{\R}$ and $B \iota^{\Or}$ induce isomorphisms on mod $2$ cohomology by Proposition \ref{max.compact.equiv}. Write $d^{\R}: \HH(\BO_n,\Z/2\Z) \to \HH^*(\BGL_n(\R),\Z/2\Z)$  and $d^{\Or}:\HH(\BO_n,\Z/2\Z) \to \HH^*(\BO_n(\C),\Z/2\Z)$ for their inverses.
  Put $w^{\R}=d^{\R}(w)$ and $w^{\Or}=d^{\Or}(w)$.

 Let $(\pi,V)$ be a real representation of a Lie group $G$, of degree $n$. Picking a basis of $V$ gives a homomorphism $\pi': G \to \GL_n(\R)$. We define
 \beq
 w(\pi):=(B \pi')^*(w^{\R}).
 \eeq
 Since a different choice of basis amounts to conjugation in $\GL_n(\R)$, we see that $w(\pi)$ is independent of this choice.
 
   If a \emph{complex} representation $(\pi,V)$ of $G$ has a real structure $\sigma$ (in the sense of \cite[page 93]{BrokerDieck}), then $\pi$ restricts to a real representation $(\pi_\sigma,V^\sigma)$ of $G$. We may define
 \beq
 w^{\R}(\pi):=w(\pi_\sigma),
 \eeq 
as it is independent of the choice of real structure.

\bigskip

 Let $V$ be a complex vector space and $\Phi$ a symmetric nondegenerate bilinear form. Let $\Or(V,\Phi)$ be the isometry group of $\Phi$. 
 By picking an orthogonal basis, we obtain an isomorphism $\Or(V,\Phi) \cong \Or_n(\C)$. We may define $w^\Phi \in \HH^*(\BO(V,\Phi),\Z/2\Z)$ to be the pullback of $w^{\Or}$ under this isomorphism.  Since picking a different orthonormal basis amounts to conjugation in $\Or_n(\C)$, the class $w^\Phi$ is independent of this choice. One checks that for $y \in \GL_{\C}(V)$, we have
 \beq
 \Int(y)^*(w^{\Phi})=w^{y^T\Phi y}.
 \eeq

A complex representation $(\pi,V)$ of $G$ is \emph{orthogonal}, provided $V$ admits a $G$-invariant symmetric nondegenerate bilinear form $\Phi$. Thus
we have $\pi: G \to \Or(V,\Phi)$. In this case define
 \beq
 w^{\Phi}(\pi):=\pi^*(w^{\Phi}).
 \eeq

In fact, if $V$ admits two such forms $\Phi,\Phi'$, then $ w^{\Phi}(\pi)= w^{\Phi'}(\pi)$; thus we may pick one and simply define $w^{\Or}(\pi)= w^{\Phi}(\pi)$.

  \begin{thm}  If $V$ admits both a $G$-invariant real structure $\sigma$, and a $G$-invariant   symmetric nondegenerate bilinear form $B$, then $w^{\Or}(\pi)=w^{\R}(\pi)$.
  \end{thm}
  
 In the main body of this paper, we drop the superscript `$\Or$', and simply write $w(\pi)$ for $w^{\Or}(\pi)$.
 Similarly we have $w_k(\pi)$ by considering the $k$th degree part of $w(\pi)$. We state some basic properties of these SWCs, omitting the proof:
  
  \begin{prop}  For a Lie group $G$, the map $w_1:\Hom(G,\mu_2) \to \HH^1(BG,\Z/2\Z)$ is an isomorphism. 
   If $(\pi,V)$ is an orthogonal representation of $G$, then $w_1(\pi)=w_1(\det \pi)$. If $\pi,\pi'$ are orthogonal representations of $G$, then $w(\pi \oplus \pi')=w(\pi) \cup w(\pi')$.
   If $(\pi,V)$ is an orthogonal representation with $\det \pi=1$, then $w_2(\pi)=0$ iff $\pi$ is spinorial.
  \end{prop}

   \bibliographystyle{alpha}
\bibliography{mybib}

\begin{thebibliography}{Hum92}

\bibitem[Ben91]{Benson}
D.~J. Benson.
\newblock {\em Representations and Cohomology: Volume 2, Cohomology of groups
  and modules}, volume~2.
\newblock Cambridge University press, 1991.

\bibitem[BH58]{Borel.Hirz.58}
A.~Borel and F.~Hirzebruch.
\newblock Characteristic classes and homogeneous spaces, i.
\newblock {\em American Journal of Mathematics}, 80(2):458--538, 1958.

\bibitem[Bor12]{Borel.LAG}
A.~Borel.
\newblock {\em Linear algebraic groups}, volume 126.
\newblock Springer Science \& Business Media, 2012.

\bibitem[Bou02]{Bou.Lie.4-6}
N.~Bourbaki.
\newblock {\em Lie groups and {L}ie algebras. {C}hapters 4--6}.
\newblock Elements of Mathematics (Berlin). Springer-Verlag, Berlin, 2002.
\newblock Translated from the 1968 French original by Andrew Pressley.

\bibitem[Bou04]{Bou.FRV}
N.~Bourbaki.
\newblock {\em Elements of mathematics: Functions of a real variable:
  Elementary theory}.
\newblock Springer Berlin, 2004.

\bibitem[Bou08]{bourbaki}
N.~Bourbaki.
\newblock {\em Lie groups and Lie algebras: chapters 7-9}, volume~3.
\newblock Springer Science \& Business Media, 2008.

\bibitem[BtD03]{BrokerDieck}
T.~Br{\"o}cker and T.~tom Dieck.
\newblock {\em Representations of compact Lie groups}, volume~98.
\newblock Springer Science \& Business Media, 2003.

\bibitem[FdV69]{Freudenthal}
H.~Freudenthal and H.~de~Vries.
\newblock {\em Linear {L}ie groups}, volume Vol. 35 of {\em Pure and Applied
  Mathematics}.
\newblock Academic Press, New York-London, 1969.

\bibitem[FH13]{Fulton}
W.~Fulton and J.~Harris.
\newblock {\em Representation theory: a first course}, volume 129.
\newblock Springer Science \& Business Media, 2013.

\bibitem[GJ23]{GJgln}
J.~Ganguly and R.~Joshi.
\newblock Total {S}tiefel {W}hitney classes for real representations of {${\rm
  GL}_n$} over {$\Bbb F_q$}, {$\Bbb R$} and {$\Bbb C$}.
\newblock {\em Res. Math. Sci.}, 10(2):Paper No. 16, 23, 2023.

\bibitem[Hat02]{hatcher}
A.~Hatcher.
\newblock {\em Algebraic topology}.
\newblock Cambridge University Press, Cambridge, 2002.

\bibitem[Hel78]{HGSN}
S.~Helgason.
\newblock Differential geometry, {L}ie groups and symmetric spaces (1978),
  1978.

\bibitem[Hum92]{RGCG}
J.~E. Humphreys.
\newblock {\em Reflection groups and Coxeter groups}.
\newblock Number~29. Cambridge University press, 1992.

\bibitem[Hus66]{Husemoller}
D.~Husem{\"o}ller.
\newblock {\em Fibre bundles}, volume~5.
\newblock Springer, 1966.

\bibitem[JS20]{joshi}
R.~Joshi and S.~Spallone.
\newblock Spinoriality of orthogonal representations of reductive groups.
\newblock {\em Representation Theory of the American Mathematical Society},
  24(15):435--469, 2020.

\bibitem[Mac98]{macdonald}
I.~G. Macdonald.
\newblock {\em Symmetric functions and Hall polynomials}.
\newblock Oxford university press, 1998.

\bibitem[MS25]{Malik.Spallone.SLn}
N.~Malik and S.~Spallone.
\newblock Stiefel-{W}hitney classes for finite special linear groups of even
  rank.
\newblock {\em J. Algebra}, 673:455--473, 2025.

\bibitem[Tod]{toda}
H.~Toda.
\newblock Cohomology of classifying spaces, homotopy theory and related topics
  (kyoto, 1984).
\newblock {\em Adv. Stud. Pure Math}, 9:75--108.

\bibitem[Wei13]{Weibel.K}
C.~A. Weibel.
\newblock {\em The $ K $-book: An Introduction to Algebraic $ K $-theory},
  volume 145.
\newblock American Mathematical Soc., 2013.

\end{thebibliography}

  \end{document}